\pdfoutput=1
\documentclass[]{siamart171218}
\usepackage[T1]{fontenc}
\usepackage[latin1]{inputenc}
\usepackage{graphicx}
\usepackage{fullpage}
\usepackage{epic}

\usepackage{subfig}

\usepackage{pgfplotstable}
\usepackage{color}
\usetikzlibrary{patterns}
\definecolor{markercolor}{RGB}{124.9, 255, 160.65}
\pgfplotsset{width=10cm,compat=1.3}
\pgfplotsset{
tick label style={font=\small},
label style={font=\small},
legend style={font=\small}
}

\newcommand{\ba}{\begin{array}}
\newcommand{\ea}{\end{array}}
\newcommand{\be}{\begin{equation}}
\newcommand{\ee}{\end{equation}}
\newcommand{\bd}{\begin{displaymath}}
\newcommand{\ed}{\end{displaymath}}
\newcommand{\bi}{\begin{itemize}}
\newcommand{\ei}{\end{itemize}}
\newcommand{\bn}{\begin{enumerate}}
\newcommand{\en}{\end{enumerate}}
\newcommand{\pa}{\partial}
\newcommand{\f}{\frac}
\newcommand{\ci}{\cite}

\newtheorem{truth}{Theorem}

\title{Hermite Methods for the Scalar Wave Equation}
\author{Daniel Appel\"{o}\thanks{Department of Applied Mathematics, University of Colorado, Boulder, CO 80309-0526
(\email{Daniel.Appelo@colorado.edu})\funding{Supported in part by NSF Grant DMS-1319054. Any conclusions 
or recommendations expressed in this paper are those of the author and do not necessarily reflect the views the NSF}}
\and
Thomas Hagstrom\thanks{Department of Mathematics,
 Southern Methodist University, Dallas, TX 75275 (\email{thagstrom@smu.edu}) \funding{Supported in part by 
NSF Grant DMS-1418871. Any conclusions or recommendations expressed in this paper are those of the author 
and do not necessarily reflect the views of the NSF}}
\and 
Arturo Vargas\thanks{Lawrence Livermore National Laboratory, 7000 East Ave., Livermore, CA 94550 (\email{vargas45@llnl.gov})}
\funding{The work has also been performed under the auspices of the U.S.
Department of Energy by Lawrence Livermore National
Laboratory under contract DE-AC52-07NA27344. IM: LLNL-JRNL-746059.}}

\date{\today}
\begin{document}
\maketitle
\begin{abstract}
Arbitrary order dissipative and conservative Hermite methods for the scalar wave equation are presented. Both methods use $(m+1)^d$ degrees of freedom per node for the displacement in $d$-dimensions; the dissipative and conservative  methods achieve orders of accuracy $(2m-1)$ and $2m$, respectively. Stability and error analyses as well as implementation strategies for accelerators are also given.
\end{abstract}



\section{Introduction}
We construct, analyze, and test arbitrary order dissipative and conservative Hermite methods for the scalar wave equation in a medium with constant speed of sound $c$. The degrees-of-freedom for
Hermite methods are tensor-product Taylor polynomials of degree $m$ in each coordinate centered
at the nodes of Cartesian grids, staggered in time. The dissipative method achieves space-time accuracy of order $2m-1$, while the conservative method has space-time order $2m$. 
Besides their high order of accuracy in both space and time combined, they have the special feature that they are stable for $c \Delta t \le h$, for \emph{all orders of accuracy}. This is significantly better than standard high-order element methods. Moreover, the large time steps are purely local to each cell, minimizing communication and storage requirements.   

Our primary interest in these schemes are as highly efficient building blocks in hybrid methods where most of the mesh can be taken to be rectilinear and where geometry is handled by more flexible (but less efficient) methods close to physical boundaries. In this work we restrict our consideration to square geometries with boundary conditions of Dirichlet, Neumann or periodic type, where boundary conditions are simple to apply. In previous work \cite{HDG} we considered this type of hybridization of the  Hermite methods for first order system proposed in \cite{GoodHer} with nodal discontinuous Galerkin (dG) methods. For wave equations in second order form we envision a similar hybridization where the geometry is handled by, for example, our recently developed discontinuous Galerkin methods for wave equations in second order form \cite{Upwind2}. Our dG method has the property that, based on the choice of numerical flux, it is either dissipative or conservative. 

We provide optimal stability and convergence results for both the conservative and dissipative method for one dimensional periodic domains. The analysis for the dissipative method follows the analysis for first order systems \cite{GoodHer} but here it is based on the energy of the wave equation $\int v^2 + | \nabla u |^2 \, dx$. A difference compared to \cite{GoodHer} is that we require that the (polynomial) approximation spaces of the velocity, $v$, and displacement, $u$ to differ by one degree.
Its extension to higher space dimensions, however, does not follow in a straightforward way from the Hermite method for first-order systems, requiring a specialized interpolation scheme to achieve order-independent stability at
CFL one.

The analysis of the conservative method is new and quite different from that of \cite{GoodHer}. Additionally, the analysis is done by introducing what we denote conserved variables, this simplifies the analysis considerably compared to the classical, \cite{Jol-2003}, way to analyze conservative methods for wave equations. To the best of our knowledge this is an original contribution.   

We also note that this paper presents the first application of Hermite methods applied to wave equations in second order form.

Many attractive high-order accurate methods have been proposed for wave equations in second order form. For example there are finite difference methods based on the summation-by-parts framework \cite{SBP} and upwinding \cite{banks2012upwind,Banks2015}, finite element methods which use mass lumping to achieve efficiency \cite{Jol-2003}, discontinuous Galerkin methods \cite{GSSwave,IPDG_Elastic,Upwind2,ChouShuXing2014}, as well as more exotic methods such as Fourier-Continuation \cite{FCAD1,FCAD2} and Galerkin differences \cite{BanksHagstromGD}.  
However, we believe that the unique properties of Hermite methods make them an interesting alternative to existing approaches.

The rest of the paper is organized as follows. In Section \ref{Prelim} we recall the basic properties of Hermite interpolation which will be used later on. In Section \ref{Diss} we describe and analyze the dissipative discretization, doing the same in Section \ref{Cons} for the conservative scheme. In Section \ref{sec:num_exp} we verify the convergence properties in one and two space dimensions and finally in Section \ref{sec:occa} demonstrate highly efficient performance of our GPU implementation.  

\section{Preliminaries}\label{Prelim}
In this section we introduce the grids we will use for the discretizations, discuss the representation of the approximations as Hermite interpolating polynomials and state some of the basic properties of Hermite interpolation. For simplicity we often restrict attention to one or two space dimensions as extensions to higher space dimensions are straightforward.  

\subsection{Grids}
Let the domain $x \in [X_L,X_R]$ be discretized by the grid
\be \label{eq:grid1D}
 x_i = X_L + i h_x, \ \ \ h_x = (X_R-X_L)/n_x, 
\ee
and let the two dimensional domain $(x,y) \in [X_L,X_R] \times [Y_B,Y_T]$ be discretized by the grid
\be \label{eq:grid2D}
 (x_i,y_j) = (X_L + i h_x, Y_B + j h_y), \ \ \ h_x = \frac{X_R-X_L}{n_x}, \ \ \ h_y = \frac{Y_T-Y_B}{n_y}.
\ee
We let the indices $i,j$ be integers or half-integers. The primal grid corresponds to
the indices $i = 0,\ldots, n_x$, $j = 0,\ldots, n_y$ and the dual grid to $i = 1/2,\ldots, n_x-1/2$, $j = 1/2,\ldots, n_y-1/2$.

In time we suppose data is defined on the primal grid when $t_n = n \Delta t$, $n=0,1,2,\ldots$, and on the dual grid when $n=1/2,3/2,5/2,\ldots$. That is, we evolve the solution between grids one half time step $\Delta t / 2$ at a time. 

\subsection{Approximation by Hermite Interpolation}
The degrees-of-freedom for our methods are approximations to the solution and its scaled derivatives at the nodes. For example,
in one space dimension the data at node $x_i$ is given by
\begin{equation*}
c_{l,i}^n \approx \f {h_x^l}{l!} \f {\pa^l u}{\pa x^l} (x_i,t_n) , \ \ \ l=0, \ldots ,m.
\end{equation*}
Our global approximation to $u(x,t_n)$ is then the piecewise degree $2m+1$ Hermite interpolant, $p^n (x)$, to the nodal data.
Precisely,
\be
p^n(x)= p_{i+1/2}^n (x) \equiv \sum_{l=0}^{2m+1} c_{l,0,i+1/2}^n \left( \f {x-x_{i+1/2}}{h_x} \right)^l ,
\ \ \ x \in [x_i,x_{i+1}] , \label{Hermn} 
\ee
where $p_{i+1/2}^n (x)$ is the unique polynomial of degree $2m+1$ satisfying for $l=0, \ldots ,m$
\be
\f {h_x^l}{l!} \f {\pa^l p_{i+1/2}^n}{\pa x^l} (x_i)=c_{l,i}^n , \ \ \  
\f {h_x^l}{l!} \f {\pa^l p_{i+1/2}^n}{\pa x^l} (x_{i+1})=c_{l,i+1}^n . \label{Hermcon}
\ee

In $d$ spatial dimensions we approximate the solution by tensor product polynomials. For example if $d=2$
the data at node
$(x_i,y_j)$ is
\begin{equation*}
c_{k,l,i,j}^n \approx \f {h_x^k h_y^l}{k! l!} \f {\pa^{k+l} u}{\pa x^k \pa y^l} (x_i,y_j,t_n) , \ \ \ k,l=0, \ldots m.
\end{equation*}
The global approximation $p^n$ is then given by
\be
p^n(x,y)= p_{i+1/2,j+1/2}^n (x,y) \equiv \sum_{k=0}^{2m+1} \sum_{l=0}^{2m+1} c_{k,l,0,i+1/2,j+1/2}^n
\left( \f {x-x_{i+1/2}}{h_x} \right)^k \left( \f {y-y_{j+1/2}}{h_y} \right)^l , \label{Hermn2}
\ee
for $(x,y) \in [x_i,x_{i+1}] \times [y_j,y_{j+1}]$ , 
where $p_{i+1/2,j+1/2}^n (x)$ is the unique tensor-product polynomial of degree $(2m+1,2m+1)$ satisfying for $k,l=0, \ldots ,m$, $i_x=i,i+1$, $j_y=j,j+1$
\be
\f {h_x^k h_y^l}{k! l!} \f {\pa^{k+l} u}{\pa x^k \pa y^l} (x_{i_x},y_{j_y},t_n) = c_{k,l,i_x j_y}^n. \label{Hermcon2}
\ee
We note that here we are using the same degrees in $x$ and $y$ to approximate $u$, but we also make use of
interpolants with differing degree, $m_x \neq m_y$, in different variables. These will be of degree $(2m_x+1,2m_y+1)$
and are defined by the obvious modifications
to (\ref{Hermn2})-(\ref{Hermcon2}). Such interpolants are also useful for adaptive computations \ci{HermitePadapt}. 

In what follows we will use the notation $\mathcal{I}_{m_1, \ldots ,m_d} f$ to denote the global piecewise polynomial Hermite interpolants defined above of some function $f$ assuming data of the degree $m_j$ in the $j$th variable.
Although these operators depend on whether data is given on the primal or dual
grid we suppress this in our notation.

Also, we note that the scaling used here is not strictly necessary but it makes the interpolation and the evaluation of the interpolants (e.g. used for dense output) slightly better conditioned as the matrix entries and coefficients do not vary too much in size. 

\subsubsection{Properties of Hermite Interpolation}

We next recall the basic properties of Hermite interpolation when applied to smooth $2\pi$-periodic functions in $d$ space dimensions. Much of the analysis of the method makes use of the seminorm $\arrowvert \cdot \arrowvert_{m_1+1,
\ldots ,m_d+1}$ defined below. Here we introduce the notation
$\mathcal{T}^d = [-\pi,\pi]^d$ and the usual multiindex notation $D^{\alpha}$, $\alpha=(\alpha_1, \ldots , \alpha_d)$ for mixed partial
derivatives of order $\alpha_j$ in the $j$th coordinate. In $d$ dimensions this semi-norm is induced by the semi-inner-product:
\begin{equation*}
\langle f,g \rangle_{m_1+1, \ldots , m_d+1} = \int_{\mathcal{T}^d} D^{\alpha} f \cdot D^{\alpha} g dV, \ \ \ 
\alpha=(m_1+1, \ldots , m_d+1) .
\end{equation*}
Note that when $d=1$ this is simply the order $m+1$ Sobolev semi-inner-product. 

A central observation is that Hermite interpolation is in fact a projection operator in this seminorm.
Precisely, as shown in \cite{GoodHer}:
\begin{equation}
  \langle \mathcal{I}_{m_1, \ldots ,m_d} u , v- \mathcal{I}_{m_1, \ldots ,m_d} v \rangle_{m_1+1, \ldots , m_d+1} =0 . \label{orthlem}
\end{equation}
From (\ref{orthlem}) we deduce the Pythagorean Theorem:
\begin{equation}
  \arrowvert u \arrowvert_{m_1+1, \ldots ,m_d+1}^2 = \arrowvert \mathcal{I}_{m_1, \ldots ,m_d} u \arrowvert_{m_1+1, \ldots ,m_d+1}^2 + \arrowvert u- \mathcal{I}_{m_1, \ldots ,m_d} u \arrowvert_{m_1+1, \ldots , m_d+1}^2 , \label{Pyth}
\end{equation}
as well as the fact that the interpolation operators are self-adjoint:
\begin{eqnarray*}
  \langle \mathcal{I}_{m_1, \ldots ,m_d} u , v \rangle_{m_1+1, \ldots , m_d+1} & = & \langle \mathcal{I}_{m_1, \ldots ,m_d} u , \mathcal{I}_{m_1, \ldots ,m_d} v \rangle_{m_1+1, \ldots , m_d+1} +
  \langle \mathcal{I}_{m_1, \ldots ,m_d} u , v- \mathcal{I}_{m_1, \ldots ,m_d} v \rangle_{m_1+1, \ldots , m_d+1} \nonumber \\ & = & \langle u , \mathcal{I}_{m_1, \ldots ,m_d} v \rangle_{m_1+1, \ldots , m_d+1} +
  \langle \mathcal{I}_{m_1, \ldots ,m_d} u -u , \mathcal{I}_{m_1, \ldots ,m_d} v \rangle_{m_1+1, \ldots , m_d+1} \nonumber \\ & = & \langle u , \mathcal{I}_{m_1, \ldots ,m_d} v \rangle_{m_1+1, \ldots , m_d+1} .
\end{eqnarray*}

Denote the $L_2$-inner product and norm by 
\bd
(f,g) = \int_{\mathcal{T}^d} f(x)^T g(x) \, dV, \ \ \ \|g \|^2 = (g,g),
\ed
and the Sobolev norms on $H^q_{\rm per}$ by 
\bd
\|f\|^2_{H^q} = \sum_{\arrowvert \alpha \arrowvert \leq q} \| D^{\alpha} \|^2 .
\ed
Then, for $d=1$, the following lemma (Lemma 3.1 and Corollary 3.1 in \cite{GoodHer}) details the basic approximation properties of Hermite interpolation. We note that these results are derived locally on each subinterval and thus generally extend to nonperiodic functions.   
\begin{lemma} \label{approx_lemma}
The Hermite interpolation operator $\mathcal{I}_m$ satisfies:
\begin{eqnarray}
\|g-\mathcal{I}_m g \| &\le& C h_x^{2m+2} \left\| \f{d^{2m+2} g }{d x^{2m+2}}  \right\| \ \    \text{for} \ \   g \in H^{2m+2}_{\rm per}, \label{approx_lemma1} \\
\| \f{d^{s}  }{d x^{s}} (g-\mathcal{I}_m g) \| &\le& C h_x^{2m+2-s} \left\| \f{d^{2m+2} g }{d x^{2m+2}}  \right\| \ \  \text{for} \ \ g \in H^{2m+2}_{\rm per}, \ \ s \leq 2m+2, \label{approx_lemma2} \\
\|g-\mathcal{I}_m g \| &\le& C h_x^{m+1} \left\| \f{d^{m+1} g }{d x^{m+1}}  \right\| \ \    \text{for} \ \   g \in H^{m+1}_{\rm per}, \label{approx_lemma3} \\
\| (f-\mathcal{I}_m f) \| &\le& C h_x^{m+1} \left\| \f{d^{m+1}  }{d x^{m+1}} (f-\mathcal{I}_m f)  \right\| \ \  \text{for} \ \ f \in H^{m+1}_{\rm per}. \label{approx_lemma4}
\end{eqnarray}  
\end{lemma}
The estimates (\ref{approx_lemma1}) - (\ref{approx_lemma3}) follow directly from the Peano kernel representation, see \cite{birkhoff1968piecewise}, and (\ref{approx_lemma4}) follows from (\ref{approx_lemma3}) with $g = f-\mathcal{I}_m f$, as for that $g$ it holds that $\mathcal{I}_m g = 0$. We note that for $d > 1$ the analogous results require the use of more complicated Sobolev seminorms
(Lemma 6.1 in \ci{GoodHer}):
\begin{equation*} 
\left\arrowvert f \right\arrowvert_{[r_1, \ldots . r_d]}^2 \equiv \sum_{\alpha \in \mathcal{Q}_r} \| D^{\alpha} f \|^2,
\end{equation*}
\begin{equation*}
\mathcal{Q}_r = \{ \alpha \arrowvert \alpha_j \in \{0,r_j\}, \ \arrowvert \alpha \arrowvert > 0 \} .
\end{equation*}

Lastly we note that the leading order error in Hermite interpolation for
sufficiently smooth functions has a particularly simple form which is the
same in all dimensions. Here for ease of notation we suppose the cell center is
the coordinate origin so that the nodes are $(\pm h/2, \ldots , \pm h/2)$.
Then, using $(x_1, \ldots ,x_d)$ to denote the coordinates and assuming equal
orders in all variables, a multi-point Taylor series expansion yields 
\be
\mathcal{I}_{m, \ldots ,m} f = f - \f {1}{(2m+2)!} \sum_{j=1}^d \left( x_j^2 - \f
{h^2}{4} \right)^{m+1} \f {\pa^{2m+2} f}{\pa x_j^{2m+2}} (0, \ldots ,0)
+ O(h^{2m+3}) .  \label{Imtrunc}
\ee

\section{Dissipative Hermite Method for the Scalar Wave Equation}\label{Diss}

We now describe our dissipative discretization for the scalar wave equation
with constant speed, $c$. 
Here we approximate the wave equation, in $d+1$ dimensions, as a first order system in time for the displacement $u$ and the velocity $v$   
\begin{eqnarray}
\f{\partial u}{\partial t} &=& v, \label{eq:wavepdeFOS1} \\
\f{\partial v}{\partial t} &=& c^2\Delta u + f, \label{eq:wavepdeFOS2}
\end{eqnarray}\label{dvWave}
with initial data
\be
u(x,0) = g_0(x), \ \ v(x,0) = g_1(x). \label{eq:wavepdeFOS3}
\ee
In this work we will consider periodic, homogenous Dirichlet, or
homogeneous Neumann boundary conditions. In one space dimension the method
is essentially equivalent to the original method in \ci{GoodHer} applied to
the $2 \times 2$ system for $(\f {\pa u}{\pa x} , v)$ followed by a reconstruction of $u$, but in multiple space dimensions additional
modifications are proposed to maintain the method's time step stability
properties. 

\subsection{One Space Dimension}
This section outlines the dissipative method in one dimension. A schematic picture of a full time step can be found in Figure \ref{fig:diss_fig}.

  \setlength{\unitlength}{0.7mm}
  \begin{figure}[htb]
    \begin{center}
      \vspace{2.5cm}
      \begin{picture}(130,60)(-10,-30)
        \thicklines
        \matrixput(27.5,2.5)(50,0){2}(0,25){3}{\circle{2}}
        \matrixput(2.5,2.5)(50,0){3}(0,25){3}{\circle*{2}}
        \put(8,0){${\cal I}$}
        \put(15,0){$\rightarrow$}
        \put(58,0){${\cal I}$}
        \put(65,0){$\rightarrow$}
        \put(33,25){${\cal I}$}
        \put(40,25){$\rightarrow$}
        \put(83,25){${\cal I}$}
        \put(90,25){$\rightarrow$}
        \put(42,0){${\cal I}$}
        \put(35,0){$\leftarrow$}
        \put(92,0){${\cal I}$}
        \put(85,0){$\leftarrow$}
        \put(17,25){${\cal I}$}
        \put(10,25){$\leftarrow$}
        \put(67,25){${\cal I}$}
        \put(60,25){$\leftarrow$}
        \put(25,8){${\cal T}$}
        \put(26,15){$\uparrow$}
        \put(75,8){${\cal T}$}
        \put(76,15){$\uparrow$}
        \put(0,33){${\cal T}$}
        \put(1,40){$\uparrow$}
        \put(50,33){${\cal T}$}
        \put(51,40){$\uparrow$}
        \put(100,33){${\cal T}$}
        \put(101,40){$\uparrow$}
        \put(0,-15){$x_{i-1}$}
        \put(25,-15){$x_{i- \frac {1}{2}}$}
        \put(50,-15){$x_{i}$}
        \put(75,-15){$x_{i+ \frac {1}{2}}$}
        \put(100,-15){$x_{i+1}$}
        \put(115,0){$t_n$}
        \put(115,25){$t_{n + \frac {1}{2}}$}
        \put(115,50){$t_{n+1}$}
      \end{picture}
         \caption{A schematic picture of the steps in the  dissipative method. Solid circles represent the primal grid and open circles represent the dual grid. ${\cal I}$ is the Hermite interpolation operator and ${\cal T}$ is the time evolution operator. \label{fig:diss_fig}}
    \end{center}
  \end{figure}
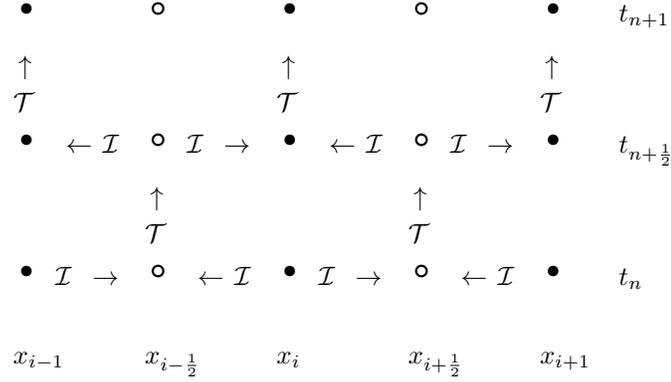

\subsubsection{Initialization}

We start by approximating $u$ and $v$ on the primal grid at the initial time, $t_0$. As discussed above, this amounts to setting
approximate values of derivatives of the functions $g_0$ and $g_1$ at the nodes:
\[
c_{l,i}^0 \approx \f{h_x^l}{l!} \f{d^l g_0}{dx^l} (x_i), \ \ l=0, \ldots m, \ \ i = 0,\ldots,n_x,
\]
\[
d_{l,i}^0 \approx \f{h_x^l}{l!} \f{d^l g_1}{dx^l} (x_i), \ \ l=0, \ldots m-1, \ \ i = 0,\ldots,n_x.
\]
Here the data $c_{l,i}^n$ will define the approximations, $p^n$, to $u$ and $d_{l,i}^n$ the approximations, $q^n$, to $v$.
Ideally we would use exact values of the derivatives but alternatively they can be obtained by interpolation or projection of the initial data. Note that derivatives of order $m$ are used in the approximation of $u$,  higher than for the approximation of $v$. The reason for this choice will be clear from the analysis presented in Section \ref{sec:analysis:diss:1d}.  

\subsubsection{Evolution}
The first step in evolving the approximate solution to the time \mbox{$t =t_0 + \Delta t / 2$} is to construct the piecewise Hermite interpolants of the data (\ref{Hermn}):
\begin{eqnarray*}
&& p_{i +\f{1}{2}}^0(x) = \sum_{l = 0}^{2m+1} c_{l,0,i+\f{1}{2}}^0 \left( \f{x - x_{i+\f{1}{2}} }{h_x}\right)^l, \ \ \forall x \in [x_{i} , x_{i+1}],\\
&& q_{i +\f{1}{2}}^0(x) = \sum_{l = 0}^{2m-1} d_{l,0,i+\f {1}{2}}^0 \left( \f{x - x_{i+\f{1}{2}} }{h_x}\right)^l, \ \ \forall x \in [x_{i} , x_{i+1}],     
\end{eqnarray*}
where the coefficients $c_{l,0,i+\f{1}{2}}$ and $d_{l,0,i+\f{1}{2}}$ are uniquely determined by the interpolation conditions (\ref{Hermcon}). 

The second step is to expand the new polynomials in time 
\bd
p_{i+\f{1}{2}}^0(x,t) = \sum_{l = 0}^{2m+1} \sum_{s = 0}^{\kappa_u(l,m)} c_{l,s,i+\f{1}{2}}^0 \left( \f{x - x_{i+\f{1}{2}} }{h_x}\right)^l \left( \f{t - t_0 }{\Delta t}\right)^s,
\ed
\bd
q_{i+\f{1}{2}}^0(x,t) = \sum_{l = 0}^{2m-1} \sum_{s=0}^{\kappa_v(l,m)} d_{l,s,i+\f{1}{2}}^0 \left( \f{x - x_{i+\f{1}{2}} }{h_x}\right)^l \left( \f{t - t_0 }{\Delta t}\right)^s.
\ed
The coefficients with $s>0$ are obtained by insisting that the approximation satisfy equations (\ref{eq:wavepdeFOS1}) and (\ref{eq:wavepdeFOS2}), and equations obtained by repeated differentiation in space and time of (\ref{eq:wavepdeFOS1}) and (\ref{eq:wavepdeFOS2}). Precisely we ask that at $(x_{i+\f{1}{2}},t_0)$ the approximations satisfy
\begin{eqnarray*}
\f{\pa^{l+s} p_{i+\f{1}{2}}^0}{\pa x^l \pa t^s} & = & \f{\pa^{l+s-1} q_{i+\f{1}{2}}^0}{\pa x^l \pa t^{s-1}},\\
\f{\pa^{l+s} q_{i+\f{1}{2}}^0}{\pa x^l \pa t^s} & = & c^2\f{\pa^{l+s+1} p_{i+\f{1}{2}}^0}{\pa x^{l+2} \pa t^{s-1}} + \f{\pa^{l+s-1} f}{\pa x^{l} \pa t^{s-1}}.
\end{eqnarray*}
Expressing these equations as truncating recursions for the coefficients we obtain
\begin{eqnarray}
c_{l,s,i+\f{1}{2}}^0 & = & \f{\Delta t}{s} d_{l,s-1,i+\f{1}{2}}^0, \label{recursion_diss_1d_1} \\
d_{l,s,i+\f{1}{2}}^0 & = & c^2 \f{(l+2)(l+1)}{s} \f{\Delta t}{h_x^2} c_{l+2,s-1,i+\f{1}{2}}^0 +\f{h_x^l}{l!} \f{\Delta t^s}{s!}\f{\pa^{l+s-1} f}{\pa x^{l} \pa t^{s-1}} . \label{recursion_diss_1d_2}
\end{eqnarray}
The upper limits of the temporal Taylor series, $\kappa_u$ and $\kappa_v$ are chosen so that, for $f=0$, the Taylor series truncate which means that $ \kappa_v(l,m) = 2m-1 - 2 \lfloor l/2 \rfloor,$ and $ \kappa_u(l,m) = \kappa_v(l,m) +1$. Here the notation $\lfloor x \rfloor$ is rounding to the nearest smaller integer ($\lfloor x \rfloor =$ ``floor of $x$''). The fact that the Taylor series truncates means that for $f=0$ the polynomial data is {\it evolved exactly}. Moreover, as long as $c \Delta t \le h_x$, the evolution at the cell center is exact for the piecewise polynomial data, a fact which is exploited in our analysis.    

Once the coefficients above have been computed we find the approximate solution at the dual nodes at the half time step
by simply evaluating the polynomials and their derivatives. The initial data for the next half time step are thus  
\begin{equation}
c_{l,i+1/2}^{1/2} = \f{h_x^l}{l!} \f{\pa^l}{\pa x^l}p_{i+\f{1}{2}}(x_{i+\f{1}{2}},t_{\f{1}{2}}), \ \ \ l = 0,\ldots,m, \label{eq:Taylorsum1}
\end{equation}
\begin{equation}
d_{l,i+1/2}^{1/2} = \f{h_x^l}{l!} \f{\pa^l}{\pa x^l}q_{i+\f{1}{2}}(x_{i+\f{1}{2}},t_{\f{1}{2}}), \ \ \ l =  0,\ldots,m-1. \label{eq:Taylorsum2}
\end{equation}

\subsubsection{Boundary Conditions for the Second Half-step}

To complete a full time step the procedure is repeated, starting with the data (\ref{eq:Taylorsum1}) and (\ref{eq:Taylorsum2}). At the interior nodes\mbox{ $x_i, i = 1,\ldots,n_x-1$} the procedure is the same as above but at the nodes $x_0$ and $x_{n_x}$ we must use the boundary conditions to fill in ghost-polynomials at $x_{-1/2}$ and $x_{n_x+1/2}$. For periodic problems this is straightforward, but for
other boundary conditions this requires additional calculations. 

Fundamentally, we must determine $2(m+1)$ unknowns that specify the polynomial at a boundary, say, at $x=x_0$. Requiring that the degree $(2m+1)$ polynomial centered at $x_0$ interpolates the $(m+1)$ data at $x_{\f{1}{2}}$ yields $(m+1)$ independent linear equations.  The first of the remaining $(m+1)$ independent linear equations can be obtained by requiring that the polynomial coincides with the boundary condition, say $u(0,t) =  g(t)$. Additional independent equations on the even derivatives of $u$ and $v$ can be obtained by differentiating the boundary condition in time together with the PDE. For example $g_{tt}(t) = u_{tt}(0,t) = u_{xx}(0,t)$ and $g_{ttv}(t) = u_{ttv}(0,t) = v_{xx}(0,t)$ and so forth.  

Practically, we ``solve'' these equations by simply extending the polynomial from the interior. For example, if the boundary condition is of Dirichlet type and constant in time we first specify the ghost polynomial so that the coefficients associated with even powers of $x$ of the interpolating polynomial (that is centered at the boundary) are all zero. If the boundary condition is non-zero we then adjust the constant term in the interpolating polynomial to agree with the value of the boundary condition. Neumann conditions are handled similarly but with an even extension of the interior polynomial to the ghost polynomial. 

We note that this approach to enforcing  boundary conditions is often referred to as compatibility boundary conditions and has been used, for example, by Henshaw \cite{henshaw:1730} in the context of solving Maxwell's equations in second order form. 

\subsection{Convergence in One Space Dimension} \label{sec:analysis:diss:1d}

We now analyze the convergence properties of the method described above. We consider $2\pi$-periodic solutions to the wave equation (\ref{eq:wavepdeFOS1})-(\ref{eq:wavepdeFOS2}) with $f=0$, for which the natural energy estimate is
\be
\f{1}{2}\f{d}{dt} \int_{-\pi}^\pi c^2 \left( \f{ \pa u}{\pa x} \right)^2 +v^2 dx = 0.
\ee
To follow the analysis in \cite{GoodHer} we will consider the above energy in the $| \cdot |_{m} $ semi-norm where the interpolation is stable and dissipative in the sense
\be
c^2 \left\arrowvert \f{ \pa \mathcal{I}_{m} u}{\pa x} \right\arrowvert_{m}^2 + \left\arrowvert \mathcal{I}_{m-1} v \right\arrowvert_{m}^2 \le c^2 \left\arrowvert \f{ \pa u}{\pa x} \right\arrowvert_{m}^2  + \left\arrowvert v \right\arrowvert_{m}^2 . \label{stable1}
\ee
This inequality is a direct consequence of (\ref{Pyth}) as $\left\arrowvert \f {\pa w}{\pa x} \right\arrowvert_m^2
= \left\arrowvert w \right\arrowvert_{m+1}^2$ for any $w(x)$. We note that this inequality is the motivation for approximating $u$
and $v$ by polynomials of differing degree. We also note that the analysis
can be applied to problems with homogeneous Dirichlet and Neumann conditions with compatible initial data by extending them smoothly to periodic problems on larger domains. 

Now, assuming that the time step satisfies $c \Delta t \le h_x$ we can write the method outlined above as an evolution of the interpolants of the basic degrees of freedom. Precisely:   
\begin{description}
\item[i.] Interpolate the initial data $p^0 = \mathcal{I}_{m} g_0$, $q^0 = \mathcal{I}_{m-1} g_1$.
\item[ii.] For $n = 0,\f{1}{2},1,\ldots, $ evolve by half steps
\begin{equation}
\left( \ba{c} p^{n+\f{1}{2}} \\ q^{n+\f{1}{2}} \ea \right) = \left( \ba{cc} \mathcal{I}_{m} & 0 \\ 0 &
\mathcal{I}_{m-1} \ea \right) \mathcal{S} \left( \ba{c} p^{n} \\ q^{n} \ea \right) .
\end{equation} 
\end{description} 
Here $\mathcal{S}$ refers to the exact evolution operators over a half time step. As noted above, for polynomial data, these coincide with the recursion relations (\ref{recursion_diss_1d_1}) and (\ref{recursion_diss_1d_2}) followed by the evaluation of the Taylor series (\ref{eq:Taylorsum1}) and (\ref{eq:Taylorsum2}) as long as the CFL condition $c \Delta t \le h_x$ is satisfied.

We now state and prove our main convergence result for the dissipative method. 
\begin{truth} \label{thm:diss1d}
Let $1 >  c \Delta t / h_x > 0$ and let $T>0$ be fixed. Suppose $g_0 \in H^{2m+1}_{\rm per}$, $g_1 \in H^{2m}_{\rm per}$ then there is a constant $C$, independent of $h_x$, so that
\begin{equation}
\left\| \frac{d}{dx} (u^n-p^n) \right\| + \| v^n-q^n \| \leq 
C h_x^{2m-1} \left(\| \frac{d^{2m+1} g_0 }{dx^{2m+1}} \| + \| \frac{d^{2m}  g_1}{dx^{2m}} \| \right) . 
\label{conv_est_diss}
\end{equation}  
\end{truth}

\begin{proof}
We start by defining the local truncation errors $\eta$ and $\xi$
\bd
\left( \ba{c} u^{n+\f{1}{2}} \\ v^{n+\f{1}{2}} \ea \right) = \left( \ba{cc} \mathcal{I}_{m} & 0 \\ 0 &
\mathcal{I}_{m-1} \ea \right) \mathcal{S} \left( \ba{c} u^{n} \\ v^{n} \ea \right) + \left( \ba{c} \eta^n \\
\xi^n, \ea \right) .
\ed
Note that, as the evolution operator $\mathcal{S}$ is exact, the truncation errors are simply the interpolation errors
\bd
\eta^n = u^{n+\f{1}{2}}- \mathcal{I}_{m} u^{n+\f{1}{2}}, \ \ \ \xi^n = v^{n+\f{1}{2}}-\mathcal{I}_{m-1} v^{n+\f{1}{2}} .
\ed
We next define the errors 
\bd
e_u^n = u^{n} - p^{n}, \ \ \ e_v^n = v^{n} - q^{n}, 
\ed
which are governed by the error equations
\be
\left( \ba{c} e_u^{n+\f{1}{2}} \\ e_v^{n+\f{1}{2}} \ea \right) = \left( \ba{cc} \mathcal{I}_{m} & 0 \\ 0 &
\mathcal{I}_{m-1} \ea \right) \mathcal{S} \left( \ba{c} e_u^{n} \\ e_v^{n} \ea \right) + \left( \ba{c} \eta^n \\
\xi^n, \ea \right) , \label{erreq_uv}
\ee
with initial data 
\bd
e_u^0 = g_0 - \mathcal{I}_{m} g_0, \ \ \ e_v^0 = g_1-\mathcal{I}_{m-1} g_1.
\ed

Now, to obtain the estimate in the theorem we will first estimate the energy error in the semi-norm where we can make use of the orthogonality (\ref{Pyth}). We start by differentiating $e_u^{n+\f{1}{2}}$ $m+1$ times and $e_v^{n+\f{1}{2}}$ $m$ times
\be
\left( \ba{c} \f {d^{m+1} e_u^{n+\f{1}{2}}}{dx^{m+1}} \\ \f {d^m e_v^{n+\f{1}{2}}}{dx^m} \ea \right) = \left( \ba{cc} \f{d^{m+1}}{dx^{m+1}}
\mathcal{I}_{m} & 0 \\ 0 & \f{d^{m}}{dx^{m}}
\mathcal{I}_{m-1} \ea \right) \mathcal{S} \left( \ba{c} e_u^{n} \\ e_v^{n} \ea \right) + \left( \ba{c} \f{d^{m+1}}{dx^{m+1}} \eta^n \\ \f{d^{m}}{dx^{m}}
\xi^n, \ea \right) . \label{erreq_uv2}
\ee
Recalling that the local truncation errors are interpolation errors we can proceed by taking the $L^2$ norms of (\ref{erreq_uv2}). Then, using the orthogonality  (\ref{Pyth}) we rewrite
\bd
\mathcal{S} \left( \ba{c} e_u^{n} \\ e_v^{n} \ea \right) = \left( \ba{c} \mathcal{S}_u E^n \\ \mathcal{S}_v E^n \ea \right)
 , \ \ E^n = \left( \ba{c} e_u^{n} \\ e_v^{n} \ea \right) , 
\ed
and arrive at 
\begin{eqnarray} 
c^2 \left\arrowvert \f{d e_u^{n+\f{1}{2}}}{dx}  \right\arrowvert^2_{m} + \left\arrowvert e_v^{n+\f{1}{2}} \right\arrowvert_{m}^2  =  c^2 \left\arrowvert \f{d\,  \mathcal{I}_{m} \mathcal{S}_u E^n}{dx}  \right\arrowvert_{m}^2 +  
\left\arrowvert  \mathcal{I}_{m-1} \mathcal{S}_v E^{n} \right\arrowvert_{m}^2  + c^2
\left\arrowvert \f{d \eta^n}{dx}   \right\arrowvert_{m}^2
+ \left\arrowvert \xi^n \right\arrowvert_{m}^2 . \label{erreq_uv3} 
\end{eqnarray}
To proceed we use the identities
\begin{eqnarray*}
\f{d^{m+1}}{dx^{m+1}}  \mathcal{S}_u E^n 
& = & \f{d^{m+1}}{dx^{m+1}}  \mathcal{I}_{m} \mathcal{S}_u E^{n} - 
\f{d^{m+1}}{dx^{m+1}}  (\mathcal{I}_{m}\mathcal{S}_u E^{n} - \mathcal{S}_u E^{n}), \\
\f{d^{m}}{dx^{m}}  \mathcal{S}_v E^{n} 
& = & \f{d^{m}}{dx^{m}}  \mathcal{I}_{m-1} \mathcal{S}_v E^{n} - 
\f{d^{m}}{dx^{m}}  (\mathcal{I}_{m-1}\mathcal{S}_u E^{n} - \mathcal{S}_v E^{n}),
\end{eqnarray*}
together with (\ref{Pyth}) to obtain:
\begin{eqnarray}
c^2 \left\arrowvert \f{d\,  \mathcal{I}_{m} \mathcal{S}_u E^n}{dx}  \right\arrowvert_{m}^2 +  
\left\arrowvert  \mathcal{I}_{m-1} \mathcal{S}_v E^{n} \right\arrowvert_{m}^2 & = &
c^2 \left\arrowvert \f{d\,  \mathcal{S}_u E^n}{dx}  \right\arrowvert_{m}^2 +  
\left\arrowvert  \mathcal{S}_v E^{n} \right\arrowvert_{m}^2 \label{interpdiss} \\
& & - c^2 \left\arrowvert \f{d\,  \left( \mathcal{S}_u E^n -\mathcal{I}_{m} \mathcal{S}_u E^n \right)}{dx}  \right\arrowvert_{m}^2 - \left\arrowvert  \mathcal{S}_v E^{n} -\mathcal{I}_{m-1} \mathcal{S}_v E^n \right\arrowvert_{m}^2 . \nonumber
\end{eqnarray}
Next we introduce the semi-norm energy error
\be
\mathcal{E}_{m}^{n} \equiv c^2 \left\arrowvert \f {d e_u^n}{dx} \right\arrowvert_m^2 + \left\arrowvert e_v^n \right\arrowvert_{m}^2 ,
\ee
and note the fundamental fact that the energy measured in any Sobolev semi-norm is preserved by the exact solution operator $\mathcal{S}$:
\be
c^2 \left\arrowvert \f{d\,  \mathcal{S}_u E^n}{dx}  \right\arrowvert_{m}^2 + \left\arrowvert  \mathcal{S}_v E^{n} \right\arrowvert_{m}^2= \mathcal{E}_{m}^{n} . \label{Energyeq} 
\ee
Now, a direct estimate in terms of  $\mathcal{E}_{m}^{n}$ does not give a sharp error estimate unless we carefully track the cancellation of errors between time-levels. To do so we first introduce $\delta^n$ by rewriting (\ref{erreq_uv3}) using (\ref{approx_lemma2})
\bd
\delta^n = \left( c^2 \left\arrowvert \f{d\,  \left( \mathcal{S}_u E^n -\mathcal{I}_{m} \mathcal{S}_u E^n \right)}{dx}  \right\arrowvert_{m}^2 + \left\arrowvert  \mathcal{S}_v E^{n} -\mathcal{I}_{m-1} \mathcal{S}_v E^n \right\arrowvert_{m}^2 \right)^{1/2}.
\ed
We then use (\ref{interpdiss})-(\ref{Energyeq}) and find:
\bd
\mathcal{E}_{m}^{n+1/2} = \mathcal{E}_{m}^{n} - \left( \delta^n \right)^2 + O(h_x^{2m}).
\ed
To expose telescoping we separate $\left( \delta^n \right)^2$ on the left
\be
\left( \delta^n \right)^2 = \mathcal{E}_{m}^{n} - \mathcal{E}_{m}^{n+1/2}  + O(h_x^{2m}) . \label{telescope}
\ee
Now, summing (\ref{telescope}) to the final time $J=T/\Delta t$ we conclude
\begin{eqnarray}
\sum_{n=0}^{J-1/2} \left( \delta^n \right)^2 & = &  \mathcal{E}_{m}^{0} - \mathcal{E}_{m}^{J}  + O(h_x^{2m-1})
\nonumber \\ & = & O(h_x^{2m-1}) , \label{deltasum}
\end{eqnarray}
so long as the initial conditions are approximated with sufficient accuracy.

To prove the theorem we must connect the estimates in the semi-norm to
$L_2$-estimates. 
Taking one derivative of $e_u^{n+\f{1}{2}}$ we have
\be
\left( \ba{c} \f {d}{dx} e_u^{n+\f{1}{2}} \\ e_v^{n+\f{1}{2}} \ea \right) = \left( \ba{c} \f{d}{dx}
\mathcal{I}_{m} \mathcal{S}_u E^n \\ 
\mathcal{I}_{m-1} \mathcal{S}_v E^n \ea \right) + \left( \ba{c} \f {d \eta^n}{dx} \\ 
\xi^n, \ea \right) . \label{erreq_uvL2}
\ee
Invoking the triangle inequality and (\ref{approx_lemma1}) and defining
\bd
\mathcal{E}^n = c^2 \| \f {d}{dx} e_u^n \|^2 + \| e_v^n \|^2
\ed
we have
\begin{eqnarray*}
\left(\mathcal{E}^{n+\f {1}{2}}\right)^{1/2} & \leq & \left( c^2 \| \f{d}{dx} \mathcal{S}_u E^n \|^2 +
\| \mathcal{S}_v E^n \|^2 \right)^{1/2} +\left( c^2 \| \f{d}{dx} \left( \left( \mathcal{I}_{m} \mathcal{S}_u
-\mathcal{S}_u \right) E^n \right) \|^2 + \| \left( \mathcal{I}_{m-1} \mathcal{S}_v - \mathcal{S}_v \right) E^n \|^2 \right)^{1/2} \\ & & + O(h_x^{2m}) .
\end{eqnarray*} 
Now using the fact that $\mathcal{S}$ preserves the energy and invoking (\ref{approx_lemma4})
\bd
\left(\mathcal{E}^{n+\f {1}{2}}\right)^{1/2} \leq \left(\mathcal{E}^{n}\right)^{1/2} + Ch^{m} \delta^n + O(h_x^{2m}).
\ed
We sum and make use of (\ref{deltasum})
\begin{eqnarray*}
\left(\mathcal{E}^{n+\f {1}{2}}\right)^{1/2} & \leq & Ch^{m} \sum_{j=0}^{n} \delta^j + O(h_x^{2m-1}) \\
& \leq & Ch^{m} \cdot \sqrt{n} \left( \sum_{j=0}^n \left( \delta^j \right)^2 \right)^{1/2} + O(h_x^{2m-1}) \\
& = & O(h_x^{2m-1}) .
\end{eqnarray*}
Finally noting that the truncation error terms driving the estimate are proportional to
\bd
c^2 \| \f {d^{2m+1} u}{dx^{2m+1}} \| + \| \f {d^{2m} v}{dx^{2m}} \|
\ed
and by the energy equality these can be bounded by their initial values 
the estimate (\ref{conv_est_diss}) follows. 
\end{proof}

\subsection{Two Space Dimensions} \label{sec:original_diss}

The straightforward extension of the method described above to two (or higher) space dimensions would be to use the extension of the one dimensional recursion relation on
tensor product interpolants based on derivative data of order $m$ in each variable to approximate $u$ and of order
$m-1$ in each variable to approximate $v$. Unfortunately such a method does not achieve stability at the full geometric CFL condition, $c \Delta t/h_{x,y} < 1$. Fortunately, we have found that a slight modification of the computation of the first time derivatives results in a method which does achieve stability at the full CFL and displays optimal convergence in the energy norm. 

The motivation for the modified method is the control of the multidimensional semi-norm in which Hermite interpolation is a projection. However, 
the modified method described below does not result in exact evolution of polynomial initial data, which is the second main ingredient in our error analysis. Therefore a more involved analysis would be required, and to date we have not been able to carry it out. 

The recursion relations for computing the time derivatives are a straightforward generalization of the one dimensional case and read as follows (for brevity we consider the case $f=0$ and suppress the indices of the space-time grid) 
\begin{eqnarray}
 c_{k,l,s} &=& \f{\Delta t}{s} d_{k, l, s-1}, \label{eq:recursion_ut} \\ 
d_{k, l, s} &=& 
c^2\f{(k+2)(k+1)}{s} \f{\Delta t}{h_x^2} c_{k+2, l, s-1} + 
c^2\f{(l+2)(l+1)}{s} \f{\Delta t}{h_y^2} c_{k,l+2,  s-1}. \label{eq:recursion_vt} 
\end{eqnarray}

Now, the local degrees of freedom representing $v$ and $u$ at each node are $m^2$ and $(m+1)^2$ coefficients of node centered tensor product polynomials of degree $m-1$ and $m$. In order for the two dimensional method to be stable with a  time step $\Delta t = {\rm CFL} \min(h_x,h_y)$ where ${\rm CFL} \approx 1$ we have found that it is necessary to start up the recursion with the following first step 
\begin{equation*}
\begin{array}{ll} 
  c_{k,l,1} = \Delta t\, d_{k, l,0}, & k,l = 0,\ldots,2m-1,\\ 
  d_{k, l, 1} = 
c^2(k+2)(k+1) \f{\Delta t}{h_x^2} {c}^X_{k+2, l} + 
c^2(l+2)(l+1) \f{\Delta t}{h_y^2} c^Y_{k,l+2}, & k,l = 0,\ldots,2m-1. 
\end{array}
\end{equation*}
Here ${c}^X_{k, l}$ are the $(2m+2) \times (2m)$ coefficients of $\mathcal{I}_{m,m-1} u$ and  
${c}^Y_{k, l}$ are the $(2m) \times (2m+2)$ coefficients of $\mathcal{I}_{m-1,m} u$. 
For the remaining coefficients $s = 2\ldots, 4m+3$ we use (\ref{eq:recursion_ut}) and (\ref{eq:recursion_vt}) with $k,l = 0,\ldots,2m+1$. 

To add up the Taylor series at the cell center we must also compute an approximation to $u$ itself. For this we
use $\mathcal{I}_{m,m}u$. As in the one dimensional case the evaluation of the approximations (now truncated in $k$ and $l$) concludes the first half-step. 

As we are considering flat boundaries we can apply the one dimensional odd/even extension of ghost polynomials to enforce Dirichlet and Neumann boundary conditions.

The motivation for the use of the special interpolants in the startup step is again the basic inequality for the conserved
energy which follows from (\ref{Pyth}):
\begin{eqnarray}
c^2 \left\arrowvert \f{ \pa \mathcal{I}_{m,m-1} u}{\pa x} \right\arrowvert_{m,m}^2 + c^2 \left\arrowvert \f{ \pa \mathcal{I}_{m-1,m} u}{\pa y} \right\arrowvert_{m,m}^2 + \left\arrowvert \mathcal{I}_{m-1,m-1} v \right\arrowvert_{m,m}^2  
&  & \nonumber \\ \leq c^2 \left\arrowvert \f{ \pa u}{\pa x} \right\arrowvert_{m,m}^2
+ c^2 \left\arrowvert \f{ \pa u}{\pa y} \right\arrowvert_{m,m}^2 + \left\arrowvert v \right\arrowvert_{m,m}^2 . & & \label{stable2}
\end{eqnarray} 
To use (\ref{stable2}) to establish stability we must supplement it with a separate treatment of functions which are
independent of one of the coordinates. This is a somewhat complex argument which we will not pursue here. In addition we have not been able to generalize the convergence estimate (\ref{conv_est_diss}). However, in all  numerical experiments (see \S \ref{sec:num_exp}) we have carried out we do observe convergence at or above the order predicted by (\ref{conv_est_diss}) as well as the ability to march in time at a CFL number one independent of order. 

We note that although the method involves four Hermite interpolants, a proper organization of the calculation based
on viewing $\mathcal{I}_{m,m-1}u$, $\mathcal{I}_{m-1,m}u$ and $\mathcal{I}_{m,m}u$ as corrections to
$\mathcal{I}_{m-1,m-1}u$ shows that for $m$ not small the cost is not much more than the cost of computing two
Hermite interpolants. 

\subsection{Truncation Error}

To better understand the error behavior of the method as well as its dispersion/dissipation properties it is useful
to examine the truncation error to leading order. To that end we use (\ref{Imtrunc}) as well as a straightforward
generalization to the mixed order interpolants used when $d>1$:
\begin{eqnarray}
\f {\pa^2 \mathcal{I}_{m, m-1} f(x,y)}{\pa x^2} & \approx \f {\pa^2 f(x,y)}{\pa x^2} & - \f {1}{(2m+2)!} \f {\pa^2}{\pa x^2} \left( x^2 - \f {h^2}{4} \right)^{m+1} \f {\pa^{2m+2} f}{\pa x^{2m+2}} (0, 0) \nonumber \\ & & 
- \f {1}{(2m)!} \left( y^2- \f {h^2}{4} \right)^m
\f {\pa^{2m+2} f}{\pa x^2 \pa y^{2m}} (0,0) , \label{Imtrunc2x} \\
\f {\pa^2 \mathcal{I}_{m-1, m} f(x,y)}{\pa y^2} & \approx \f {\pa^2 f(x,y)}{\pa y^2} & - \f {1}{(2m+2)!} \f {\pa^2}{\pa y^2}
\left( y^2 - \f {h^2}{4} \right)^{m+1} \f {\pa^{2m+2} f}{\pa y^{2m+2}} (0, 0) \nonumber \\ & &
- \f {1}{(2m)!} \left( x^2 - \f {h^2}{4}
\right)^m \f {\pa^{2m+2} f}{\pa x^{2m} \pa y^2} (0,0). \label{Imtrunc2y}
\end{eqnarray}
Again we assume the cell center is the origin. Writing $c \Delta t = \lambda h$ the exact update formulas are
\begin{eqnarray*}
u(0,0,\Delta t/2) & = & \sum_{j=0}^{m+1} \f {\left( \lambda \f {h}{2} \right)^{2j}}{(2j)!} \nabla^{2j} u(0,0,0) +
c^{-1} \sum_{j=0}^m \f {\left( \lambda \f {h}{2} \right)^{2j+1}}{(2j+1)!} \nabla^{2j} v(0,0,0) + O(h^{2m+3}) , \\
v(0,0,\Delta t/2) & = & c \sum_{j=1}^{m+1} \f {\left( \lambda \f {h}{2} \right)^{2j-1}}{(2j-1)!} \nabla^{2j} u(0,0,0) +
\sum_{j=0}^m \f {\left( \lambda \f {h}{2} \right)^{2j}}{(2j)!} \nabla^{2j} v(0,0,0) + O(h^{2m+2}) .
\end{eqnarray*}
Substituting (\ref{Imtrunc}) as well as (\ref{Imtrunc2x})-(\ref{Imtrunc2y}) into the update formulas for the
discrete solutions $u^{1/2}$, $v^{1/2}$ we derive an expression for the truncation errors to leading order. To
simplify these we introduce:
\bd
L^k u \equiv \f {\pa^k u}{\pa x^k} (0,0,0) + \f {\pa^k u}{\pa y^k}(0,0,0), \ \ M^k u \equiv 
\f {\pa^k u}{\pa x^{k-2} \pa y^2} (0,0,0) 
+ \f {\pa^k u}{\pa x^2 \pa y^{k-2}} (0,0,0).
\ed
Then 
\begin{eqnarray*}
u(0,0,\Delta t/2) - u^{1/2}(0,0) & = & \f {1}{(2m+2)!} L^{2m+2} u \sum_{j=0}^{m+1} \f {\left( \lambda \f {h}{2} \right)^{2j}}{(2j)!}
(2j)! \left( \ba{c} m+1 \\ j \ea \right) \left( - \f {h^2}{4} \right)^{m+1-j} \\ & & + \f {1}{(2m)!} M^{2m+2} u
\sum_{j=1}^{m+1} \f {\left( \lambda \f {h}{2} \right)^{2j}}{(2j)!} (2j-2)! \left( \ba{c} m \\ j-1 \ea \right) \left( - \f {h^2}{4} \right)^{m+1-j} \\ & & + \f {c^{-1}}{(2m)!} L^{2m}v   \sum_{j=0}^m \f {\left( \lambda \f {h}{2} \right)^{2j+1}}{(2j+1)!} (2j)! \left( \ba{c} m \\ j \ea \right) \left( - \f {h^2}{4} \right)^{m-j} + O(h^{2m+3}) , \\
v(0,0,\Delta t/2) - v^{1/2}(0,0) & = & \f {1}{(2m+2)!} L^{2m+2}u  \sum_{j=1}^{m+1} \f {\left( \lambda \f {h}{2} \right)^{2j-1}}{(2j-1)!} (2j)! \left( \ba{c} m+1 \\ j \ea \right) \left( - \f {h^2}{4} \right)^{m+1-j} \\ & & + \f {1}{(2m)!}
M^{2m+2}u  \sum_{j=1}^{m+1} \f {\left( \lambda \f {h}{2} \right)^{2j-1}}{(2j-1)!} (2j-2)! \left( \ba{c} m \\ j-1 \ea \right) \left( - \f {h^2}{4} \right)^{m+1-j} \\ & & + \f {1}{(2m)!} L^{2m} v \sum_{j=0}^m \f {\left( \lambda \f {h}{2} \right)^{2j}}{(2j)!} (2j)! \left( \ba{c} m \\ j \ea \right) \left( - \f {h^2}{4} \right)^{m-j} + O(h^{2m+2}) .
\end{eqnarray*}
To simplify these expressions introduce some notation
and apply the binomial formula to find
\begin{multline} 
u(0,0,\Delta t/2) - u^{1/2}(0,0)  \\
= \lambda \f {h}{2} \left( \alpha h^{2m+1} L^{2m+2}u + \beta h^{2m+1} M^{2m+2} u
+c^{-1} \gamma h^{2m} L^{2m} v + O(h^{2m+2}) \right), \label{utrunc} 
\end{multline}
\begin{multline} 
v(0,0,\Delta t/2) - v^{1/2}(0,0) \\  = \lambda \f {h}{2} \left( c \eta h^{2m} L^{2m+2}u +  c \gamma h^{2m} M^{2m+2} u + 
\nu h^{2m-1} L^{2m}v + O(h^{2m+1}) \right), \label{vtrunc} 
\end{multline}
where
\bd
\nu = \f {2^{-(2m-1)}}{(2m)!} \f {\left( \lambda^2-1 \right)^m}{\lambda} , \ \ \alpha=\f {\lambda^2-1}{4(2m+2)(2m+1)} \nu,  \ \
\eta = \f {\lambda}{4m+2} \nu,
\ed
\bd
\beta= \f {2^{-(2m+2)}}{(2m)!} \sum_{j=0}^m (j+1)^{-1} (2j+1)^{-1} \lambda^{2j+1} (-1)^{m-j}
\left( \ba{c} m \\ j \ea \right),  \ \  
\gamma  =  \f {2^{-2m}}{(2m)!} \sum_{j=0}^m (2j+1)^{-1} \lambda^{2j} (-1)^{m-j} \left( \ba{c} 
m \\ j \ea \right) .
\ed 

From (\ref{utrunc})-(\ref{vtrunc}) we make the following predictions which will be verified in the numerical
experiments:
\begin{description}
\item[i.] For $\lambda < 1$ the third term in (\ref{vtrunc}) dominates the error and we expect convergence at order $2m-1$, exactly as proved for the energy when $d=1$.
\item[ii.] For $\lambda \rightarrow 1$ the second term in (\ref{vtrunc}) dominates the error and the convergence rate will approach $2m$.
\end{description} 

\section{Conservative Hermite Methods for the Wave Equation}\label{Cons}

In contrast with the dissipative method, our conservative method only uses the variable $u$, directly approximating
\begin{equation}
\f {\pa^2 u}{\pa t^2} = c^2 \nabla^2 u +f . \label{wavequ}
\end{equation}

The starting point for the conservative schemes is the Taylor expansion in time of $u(x,t)$ around \mbox{$t=t_n + \f{\Delta t}{2}$} and \mbox{$t=t_n - \f{\Delta t}{2}$}, i.e.
\bd
u(x, t \pm \f{\Delta t}{2}) = \sum_{r=0} \f{1}{r!} \left(\pm \f{\Delta t}{2}\right)^{r} \f{\pa^{r} u(x, t_n)}{\pa t^{r}}.
\ed
Adding the equations we find 
\be \label{eq:update}
u(x,t+\f{\Delta t}{2}) + u(x,t-\f{\Delta t}{2}) = 2 \sum_{l=0} \f{1}{2l!} \left( \f{\Delta t}{2}\right)^{2l} \f{\pa^{2l} u(x,t_n)}{\pa t ^{2l}} = 2 \sum_{l=0} \f{1}{2l!} \left( \f{c \Delta t}{2}\right)^{2l} \nabla^{2l} u(x,t_n),
\ee
where we have used equation (\ref{wavequ}) with $f=0$ to replace time derivatives with derivatives in $x$. With these formulas we pursue a Hermite-based generalization of high-order conservative time-stepping schemes
based on the so-called modified equation approach; see, for example, \ci{henshaw:1730} for an example in the finite difference context.

  \setlength{\unitlength}{0.7mm}
  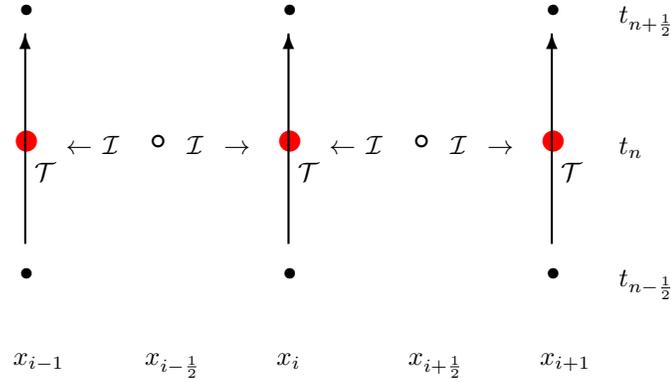
\begin{figure}[htb]
    \begin{center}
      \vspace{2.5cm}
      \begin{picture}(130,60)(-10,-30)
        \thicklines
        \matrixput(27.5,27.5)(50,0){2}(0,50){1}{\circle{2}}
        \matrixput(2.5,2.5)(50,0){3}(0,50){2}{\circle*{2}}
        \matrixput(2.5,27.5)(50,0){3}(0,50){1}{\color{red}{\circle*{4}}}


        \put(33,25){${\cal I}$}
        \put(40,25){$\rightarrow$}
        \put(83,25){${\cal I}$}
        \put(90,25){$\rightarrow$}
        \put(17,25){${\cal I}$}
        \put(10,25){$\leftarrow$}
        \put(67,25){${\cal I}$}
        \put(60,25){$\leftarrow$}
        \put(4,20){${\cal T}$}
        \put(2.25,8){\vector(0,1){40}}

        \put(54,20){${\cal T}$}
        \put(52.25,8){\vector(0,1){40}}
        \put(104,20){${\cal T}$}
        \put(102.25,8){\vector(0,1){40}}

        \put(0,-15){$x_{i-1}$}
        \put(25,-15){$x_{i- \frac {1}{2}}$}
        \put(50,-15){$x_{i}$}
        \put(75,-15){$x_{i+ \frac {1}{2}}$}
        \put(100,-15){$x_{i+1}$}
        \put(115,0){$t_{n - \frac {1}{2}}$}
        \put(115,25){$t_{n }$}
        \put(115,50){$t_{n+ \frac {1}{2}}$}
      \end{picture}
         \caption{A schematic picture of the steps in the conservative method. Solid black circles represent the primal grid and open circles represent the dual grid. The red circles represent the interpolated solution where the equations are enforced. ${\cal I}$ is the Hermite interpolation operator and ${\cal T}$ is the time evolution operator. \label{fig:cons_fig}}
    \end{center}
  \end{figure}

\subsection{The Method in One Space Dimension}

Using the same grids and degrees-of-freedom, $c_{k,i}^n$, as in the previous section,
as well as the Hermite interpolant at the current time level
\bd
p_i^{n}(x) = \sum_{l=0}^{2m+1} c_{l,0,i}^n \left( \f {x-x_i}{h_x} \right)^l
\ed
we may employ the update formula (\ref{eq:update}) to find $c_{k,i}^{n+1/2}$
\be
c_{k,i}^{n+\f{1}{2}} + c_{k,i}^{n-\f{1}{2}} = 2 \sum_{l=0}^{m - \lfloor \f{k}{2} \rfloor } \left( \f{c\Delta t}{2h_x}\right)^{2l} c_{2l+k,0,i}^{n} \f{(2l+k)!}{(2l)! \,k!}, \ \ k = 0,\ldots,m. \label{eq:convUpdate} 
\ee
Here the truncation of the sum guarantees that all terms are included and hence the evolution of polynomial data is exact. Thus, under the usual CFL condition $c \Delta t < h_x$, the updated data $c_{k,i}^{n+\f{1}{2}}$ is also exact for the evolved piecewise polynomial approximations, $p_i^n$ and $p_i^{n-1/2}$. Boundary conditions are enforced in the same way as for the dissipative method. A schematic of the grids used and the process for taking one step can be found in Figure \ref{fig:cons_fig}.

\subsection{Convergence in One Space Dimension}

Assuming periodic boundary conditions we write the method as:
\begin{equation}
p^{n+1/2} = \mathcal{I} \left(\mathcal{S}_{+} p^{n} + \mathcal{S}_{-} p^{n} \right) -p^{n-1/2} . \label{eqn}
\end{equation} 
The half-step solution operators $\mathcal{S}_{\pm}$ are defined by:
\bd
\mathcal{S}_{\pm}w (x) = w \left( x \pm c \f {\Delta t}{2} \right),
\ed
and as above $\mathcal{I}$ denotes both the dual-to-primal and primal-to-dual
degree $2m+1$ Hermite interpolation operators with cell widths $h$ (we are suppressing the suffixes as the interpolants are always degree $2m+1$ for the conservative methods).

Concerning the shift operators $\mathcal{S}_{\pm}$ we note that:
\begin{eqnarray}
\arrowvert \mathcal{S}_{\pm} f \arrowvert_{m+1} & = & \arrowvert f \arrowvert_{m+1} , \label{Sem} \\
\| \mathcal{S}_{\pm} f \| & = & \| f \| , \label{Se}
\end{eqnarray}
The exact solution, $u$, satisfies:
\be
u(x,t+\Delta t/2) = \mathcal{S}_{+}u(x,t)+ \mathcal{S}_{-}u(x,t) - u(x,t-\Delta t/2). \label{exact2level}
\ee
The error, $e^n(x)=u(x,t_n)-p^n(x)$, then satisfies
\begin{eqnarray}
e^{n+1/2} & = & \mathcal{I} \left(\mathcal{S}_{+} e^{n} + \mathcal{S}_{-} e^{n} \right) \label{erreqn} \\
& & + (1-\mathcal{I}) \left(\mathcal{S}_{+} u(\cdot,t_n) + \mathcal{S}_{-} u(\cdot,t_n) \right) - e^{n-1/2} . \nonumber  
\end{eqnarray}

We now introduce conserved variables:
\bd
U_{\pm}(x,t)= u(x,t)-\mathcal{S}_{\pm} u(x,t-\Delta t/2) , 
\ed
and note the equation
\be
U_{\pm}(x,t+\Delta t/2)=\mathcal{S}_{\mp} U_{\pm}(x,t) , \label{Uev}
\ee
holds as a result of (\ref{exact2level}). The equation obviously implies using (\ref{Sem}),(\ref{Se}) that the $L^2$ norm and Sobolev seminorms of $U_{\pm}$ are conserved in time.
Similarly we define
\be
P_{\pm}^n  = p^n  - \mathcal{S}_{\pm} p^{n-1/2} , \label{Vdef}
\ee
and note that, using the fact that $\mathcal{I}p^n = p^n$, and the definition of the method (\ref{eqn}), 
\begin{equation}
P_{\pm}^{n+1/2} = \mathcal{I} \mathcal{S}_{\mp} P_{\pm}^{n} + (\mathcal{I}-1)  S_{\pm} P_{\mp}^{n} . \label{Vev} 
\end{equation} 
Setting
\be
E_{\pm}^n(x) = U_{\pm}(x,t_n) - P_{\pm}^n(x) , 
\ee
note that
\be
E_{\pm}^n = e^n  - \mathcal{S}_{\pm} e^{n-1/2} , 
\label{Eeeq}
\ee
and calculate directly using (\ref{Uev}), (\ref{Vev}):
\begin{eqnarray}
E_{\pm}^{n+1/2} & = & \mathcal{I} \mathcal{S}_{\mp} E_{\pm}^{n} + (\mathcal{I}-1) \mathcal{S}_{\pm} E_{\mp}^{n} \nonumber \\
& & +(1-\mathcal{I}) \left( \mathcal{S}_{+} U_{-}(\cdot,t_{n}) + \mathcal{S}_{-} U_{+}(\cdot,t_{n}) \right) . \label{Eev} 
\end{eqnarray}
Note that the absence of $\pm$ in the last term due to symmetry.

\subsubsection{Energy and error estimates in the Sobolev seminorm}

We first note that just as the norms (and seminorms) of $U_{\pm}$ are conserved by the continuous evolution the sum of the seminorms of
$P_{\pm}^{n}$ are also conserved. 
By the orthogonality property of the Hermite interpolation operator (\ref{orthlem}) and the seminorm-preserving property of
$\mathcal{S}_{\pm}$ we have using (\ref{Vev}) 
\begin{eqnarray}
\left\arrowvert P_{+}^{n+1/2} \right\arrowvert_{m+1}^2 + \left\arrowvert P_{-}^{n+1/2} \right\arrowvert_{m+1}^2 & = &
\left\arrowvert \mathcal{I} \mathcal{S}_{-} P_{+}^{n} \right\arrowvert_{m+1}^2 + 
\left\arrowvert (1-\mathcal{I}) \mathcal{S}_{-} P_{+}^{n} \right\arrowvert_{m+1}^2 \nonumber \\ & & + 
\left\arrowvert \mathcal{I} \mathcal{S}_{+} P_{-}^{n} \right\arrowvert_{m+1}^2 + 
\left\arrowvert (1-\mathcal{I}) \mathcal{S}_{+} P_{-}^{n} \right\arrowvert_{m+1}^2 \nonumber \\ & = &
\left\arrowvert \mathcal{S}_{-} P_{+}^{n} \right\arrowvert_{m+1}^2 + 
\left\arrowvert \mathcal{S}_{+} P_{-}^{n} \right\arrowvert_{m+1}^2 \nonumber \\ & = &
\left\arrowvert P_{+}^{n} \right\arrowvert_{m+1}^2 + \left\arrowvert P_{-}^{n} \right\arrowvert_{m+1}^2 , \label{Vconst}
\end{eqnarray}
with the analogous equality holding at the next half step.

We now consider (\ref{Eev}). For simplicity we define
\be
\Delta^n = \mathcal{S}_{+} U_{-}(\cdot,t_{n}) + \mathcal{S}_{-} U_{+}(\cdot,t_{n}) ,
\ee
Note that
\begin{eqnarray}
U_{\pm} (x,t) & = & u(x,t) - u(x \pm c \f {\Delta t}{2} , t- \f {\Delta t}{2} ) \nonumber \\
& = &  \f {\Delta t}{2} \left( \f {\pa u}{\pa t} (x,t) \pm c \f {\pa u}{\pa x} (x,t) \right) + O (\Delta t^2) .
\end{eqnarray}
Thus for sufficiently smooth solutions $u$ we have
\begin{equation}
\left\arrowvert (1-\mathcal{I}) \Delta^{n} \right\arrowvert_{m+1} \leq C \Delta t \cdot h^{m+1} . \label{truncm} 
\end{equation}
We then have:
\begin{eqnarray}
\left\arrowvert E_{+}^{n+1/2} \right\arrowvert_{m+1}^2 + \left\arrowvert E_{-}^{n+1/2} \right\arrowvert_{m+1}^2 & = &
\left\arrowvert \mathcal{I} \mathcal{S}_{-} E_{+}^{n} \right\arrowvert_{m+1}^2 + 
\left\arrowvert (1-\mathcal{I}) \mathcal{S}_{-} E_{+}^{n} \right\arrowvert_{m+1}^2 \nonumber \\ & & +   
\left\arrowvert \mathcal{I} \mathcal{S}_{+} E_{-}^{n} \right\arrowvert_{m+1}^2 + 
\left\arrowvert (1-\mathcal{I}) \mathcal{S}_{+} E_{-}^{n} \right\arrowvert_{m+1}^2 \nonumber \\ & &
-2 \langle (1-\mathcal{I}) \left( \mathcal{S}_{+} E_{-}^{n} + \mathcal{S}_{-} E_{+}^{n} \right), (1-\mathcal{I}) \Delta^{n} \rangle_{m+1} 
\nonumber \\ & & + 2  \left\arrowvert (1-\mathcal{I}) \Delta^{n} \right\arrowvert_{m+1}^2 \nonumber \\ & \leq &
\left\arrowvert E_{+}^{n} \right\arrowvert_{m+1}^2 + \left\arrowvert E_{-}^{n} \right\arrowvert_{m+1}^2  \label{Eineq} \\ & & 
+ C \left( \left\arrowvert E_{+}^{n} \right\arrowvert_{m+1}^2 + \left\arrowvert E_{-}^{n} \right\arrowvert_{m+1}^2 \right)^{1/2}
\Delta t \cdot h^{m+1} + C \Delta t^2 \cdot h^{2m+2}, \nonumber
\end{eqnarray}
where we have made use of the orthogonality relation (\ref{orthlem}) and the Cauchy-Schwarz inequality.

Defining
\be
\mathcal{E}^{n}_{m+1} = \left\arrowvert E_{+}^{n} \right\arrowvert_{m+1}^2 + \left\arrowvert E_{-}^{n} \right\arrowvert_{m+1}^2 , 
\ee
and summing over all steps we have 
\be
\mathcal{E}^{n}_{m+1} \leq \mathcal{E}^{1/2}_{m+1} + C \Delta t \cdot h^{m+1} \left[ \sqrt{\mathcal{E}^{1/2}_{m+1}} 
+ \sum_{j=1}^{n-1} \left( \sqrt{\mathcal{E}^j_{m+1}} + \sqrt{\mathcal{E}^{j+1/2}_{m+1}} \right) \right] + C t_n \Delta t \cdot h^{2m+2} .
\label{sumEineq} 
\ee
Note that $\mathcal{E}^{1/2}_{m+1}$ is directly determined by the initial data and the first half time step, which we assume has been
determined with sufficient accuracy that
\be
\mathcal{E}^{1/2}_{m+1} \leq C h^{m+1} .
\ee
Setting
\be
\mathcal{M}_{m+1}^{n} = \max_{j \leq n} \mathcal{E}_{m+1}^j ,
\ee
we have
\be
\mathcal{M}_{m+1}^{n} \leq Ct_n h^{m+1} \sqrt{\mathcal{M}_{m+1}^n} + C h^{2m+2} (1+ t_n \Delta t) .
\ee
From this we conclude the fundamental error estimate in the seminorm.

\begin{lemma}\label{semilem}
For $u$ sufficiently smooth there exists $C$ independent of $h$ and $\Delta t$ such that
\be
\mathcal{E}_{m+1}^{n}  \leq Ch^{2m+2} (1+t_n^2) . \label{Ebnd} 
\ee
\end{lemma}

\subsubsection{Error estimates in $L^2$}

We begin by estimating the $L^2$ norms of the errors for the conserved variables, $E^{n}_{\pm}$. Using (\ref{Eev}) in conjunction with (\ref{approx_lemma1}), (\ref{approx_lemma4}), (\ref{Se}) and (\ref{Ebnd}) we obtain the inequality 
\begin{eqnarray}
\| E_{\pm}^{n+1/2} \| & \leq & \| S_{\mp} E_{\pm}^{n} \| + \| (\mathcal{I}-1) S_{\mp} E_{\pm}^{n} \| + \| (\mathcal{I}-1) S_{\pm} E_{\mp}^{n} \|  + \| (\mathcal{I}-1) \Delta^{n} \| \nonumber \\ & \leq & \| E_{\pm}^{n} \| + Ch^{m+1} \sqrt{\mathcal{E}_{m+1}^{n}}
+C  \Delta t \cdot h^{2m+2} \nonumber \\
& \leq & \| E_{\pm}^{n} \| + C(1+t_n) h^{2m+2} . \label{E2ineq} 
\end{eqnarray}
Assuming the initial data is sufficiently accurate
\be
\| e^{0,1/2} \| \leq Ch^{2m+1} ,
\ee
and that the CFL number is bounded below
\be
c \f {\Delta t}{h} \geq \eta > 0,
\ee
summing (\ref{E2ineq}) yields:
\be
\| E_{\pm}^{n} \| \leq C (1+t_n^2) h^{2m+1} . \label{E2bnd} 
\ee

Now from (\ref{Eeeq}) and (\ref{E2bnd}) we obtain
\begin{equation}
\| e^{n+1/2} \| \leq \| S_{\pm} e^{n} \| + \| E_{\pm}^{n+1/2} \|  \leq \| e^{n} \| + C (1+t_n^2) h^{2m+1} . \label{e2n}   
\end{equation}
Summing (\ref{e2n}) we obtain our main convergence result.

\begin{truth}\label{cconthm}
For $u$ sufficiently smooth there exists $C$ independent of $h$ and $\Delta t \geq c^{-1} \eta h$ such that
\be
\| e^n \| \leq Ch^{2m} (1+t_n^3) . \label{errbnd} 
\ee
\end{truth} 

\subsection{The Method in Two Space Dimensions}

Again using (\ref{eq:update}) and, as for the dissipative method, using the notation in (\ref{Hermn2}), (\ref{Hermcon2}),
we derive an explicit and compact update formula for the coefficients (again the spatial indices are suppressed)
\be
c_{k,l}^{n+\f{1}{2}} + c_{k,l}^{n-\f{1}{2}} = \f{2}{k!\, l!} \sum_{r=0}^{2m} \sum_{\{i_x,i_y\} \in \Sigma} (k+2 i_x)! \, (l+2i_y)! \mathcal{P}_{i_x,i_y} c_{k+2i_x,l+2i_y,0}^{n}, \ \ k,l = 0,\ldots,m.
\ee
Here $\Sigma$ is the set of integers that satisfy all of the relations 
\[
i_x + i_y = r, \ \ (k+2i_x) \le (2m+1), \ \ (l+2i_y) \le (2m+1),
\]
or said more plainly, we make sure to not address outside the array holding $c_{(\cdot ,\cdot,0)}^n$. The matrix $\mathcal{P}$ is a scaled version of Pascal's triangle (oriented to start in element $\mathcal{P}_{0,0}$). Precisely, let 

\be
P = \left(
 \begin{array}{rrrrr}
1 & 1 &  1 & 1 &\ldots \\
1 & 2 &  3 & 4 & \ldots \\
1 & 3 &  6 & 10 & \ldots \\
1 & 4 &  10 & \ddots & \ldots \\
\vdots & \vdots & \vdots & \vdots &\ddots 
\end{array}	
\right),
\ee
then 
\[
\mathcal{P}_{i,j} = \f{1}{(2i+2j)!} P_{i,j} \left(\f{c\Delta t}{h_x}\right)^{2i} \left(\f{c\Delta t}{h_y}\right)^{2j},
\]
where the indexation for $i$ and $j$ starts with zero. Obviously, extensions to more space dimensions are straightforward. 

\subsection{Truncation Error}

The truncation error for the conservative method is computed as in the dissipative case. In particular, noting that
truncating the sum in (\ref{eq:update}) at $l=m+1$ produces an error of $O(h^{2m+4})$, setting $c \Delta t = \lambda h$,
and invoking (\ref{Imtrunc}) we find:
\begin{eqnarray}
u(0,0,\Delta t)-p^1(0,0) & = & \left( \f {c \Delta t}{2} \right)^2 \left( h^{2m} \f {2^{-(2m+1)}}{(2m+2)!} \lambda^{-2} L^{2m+2}u
\sum_{l=0}^{m+1} \lambda^{2l} (-1)^{m+1-l} \left( \ba{c} m+1 \\ l \ea \right) + O(h^{2m+2}) \right) \nonumber \\
& = & \left( \f {c \Delta t}{2} \right)^2 \left( h^{2m} \f {2^{-(2m+1)}}{(2m+2)!} \f {(\lambda^2-1)^{m+1}}{\lambda^{2}} L^{2m+2}u
+ O(h^{2m+2}) \right) . \label{Ctrunc}
\end{eqnarray}

From (\ref{Ctrunc}) we predict, as proven for $d=1$, that the convergence rate for the conservative method will be $2m$,
approaching $2m+2$ as $\lambda \rightarrow 1$. This will be verified in the numerical experiments.

\section{Numerical Experiments}\label{sec:num_exp}
\subsection{A Gaussian Pulse in a 1D Box}
The aim of this example is to empirically investigate the rates of convergence of the conservative and dissipative methods as a function of $m$. To do this we solve 
\[
u_{tt}=u_{xx},
\]
on $x \in [-3/2,3/2]$ with the initial data 
\[
u(x,0) = e^{-20x^2}, \ \ v(x,0) = 0,
\]
and with boundary conditions $u(-3/2,t) = 0$, $u(3/2,t) = 0$. The boundary conditions are enforced as described above, that is, if we denote the $m+1$ coefficients of the polynomial approximating $u$ at the leftmost dual gridpoint $x_{1/2} = -3/2+h_x/2$ by $c_l[x_{1/2}]$, we set $c_l[x_{-1/2}] = (-1)^l c_l[x_{1/2}]$. Similarly on the right boundary we set the ``ghost-polynomial'' $c_l[x_{n_x+1/2}] = (-1)^{(l+1)} c_l[x_{n_x-1/2}]$. The boundary conditions for the velocity are handled in the same way.

We solve until time $12+\tau$, $\tau \approx .25$, when the exact solution is
\[
u(x,12+\tau) = \f{1}{2}\left( e^{-20 (x+\tau)^2} + e^{-20 (x-\tau)^2} \right),
\]
and compute the $L_2$-error for the solution evaluated on a sufficiently fine grid. For all $m$ we choose the time step so that  $\Delta t / h_x  = \lambda $. We use a sequence of grid refinements starting with $n_x^{(0)} = 10$ and $n_x^{(i)} = \lceil 1.2 n_x^{(i-1)} \rceil$ and fix $\lambda = 0.8$ and $1.0$. In order to be able to take an integer number of time steps we adjust $\tau$ to be as close to $.25$ as possible.  

The $L^2$-errors for resulting $h_x$ and for the dissipative method are displayed in Figure \ref{fig:diss_ex_1_errors} and the results for the conservative method can be found in Figure \ref{fig:cons_ex_1_errors}. In both figures we display results for $m = 0,\ldots,5$. As can be seen the rates of convergence are $(2m-1)$ and $(2m)$ for the dissipative method for $\lambda = 0.8$ and $1.0$ and $(2m)$ and $(2m+2)$ for the conservative method. These are as predicted by the truncation error analysis given above. 

\begin{figure}[htb]
  \begin{center}
  \includegraphics[width=0.45\textwidth]{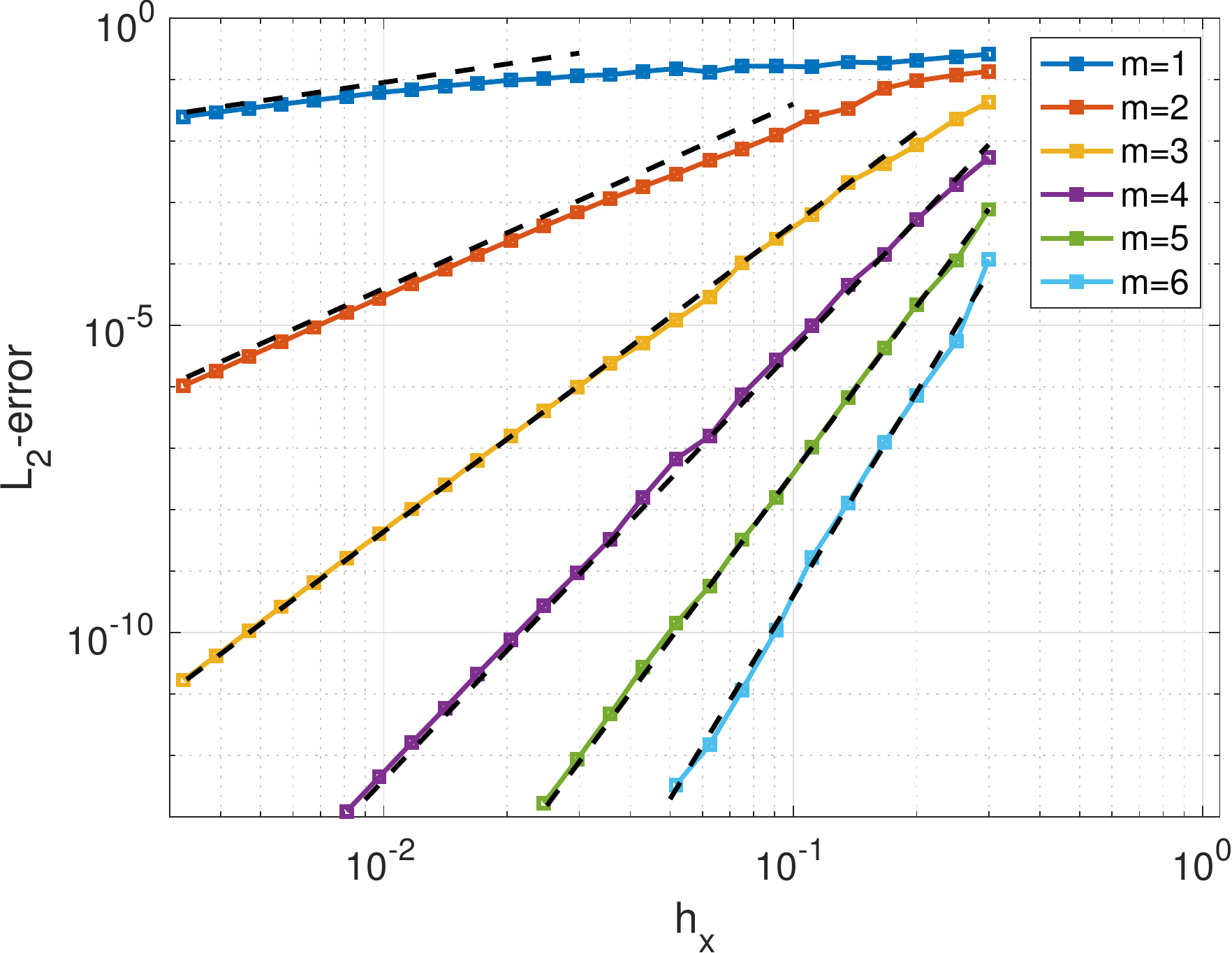}
  \includegraphics[width=0.45\textwidth]{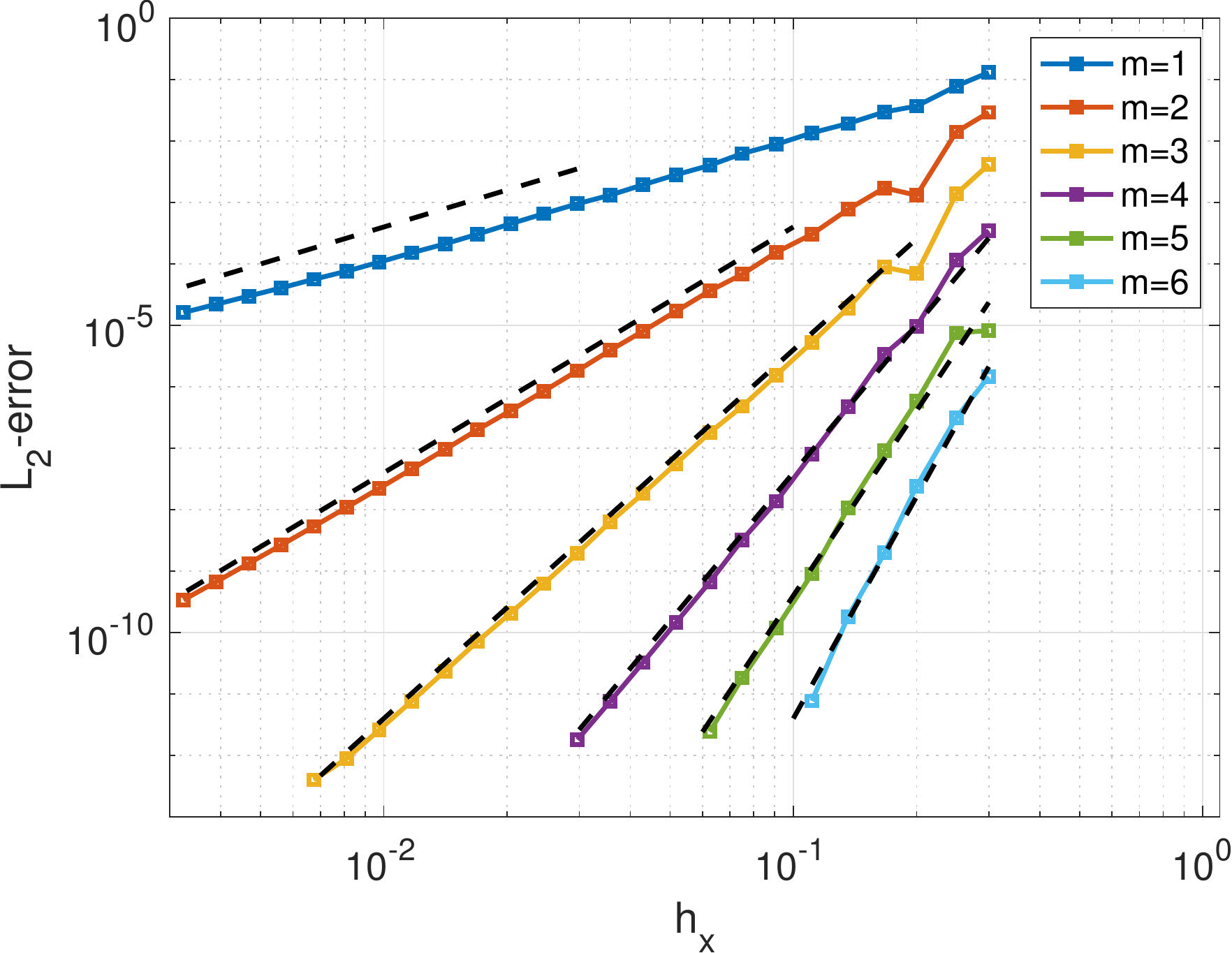}
  \caption{The dissipative method, results are for the example with the Gaussian initial data. Displayed are $L^2$-errors as a function of $h_x$ for $m = 1,\ldots,6$. To the left is for $\lambda = 0.8$ and to the right   $\lambda = 1.0$. The dashed lines are $\sim  h_x^{2m-1}$ to the left and $\sim  h_x^{2m}$ to the right. \label{fig:diss_ex_1_errors}}
  \end{center}
\end{figure}

\begin{figure}[htb]
\begin{center}
  \includegraphics[width=0.45\textwidth]{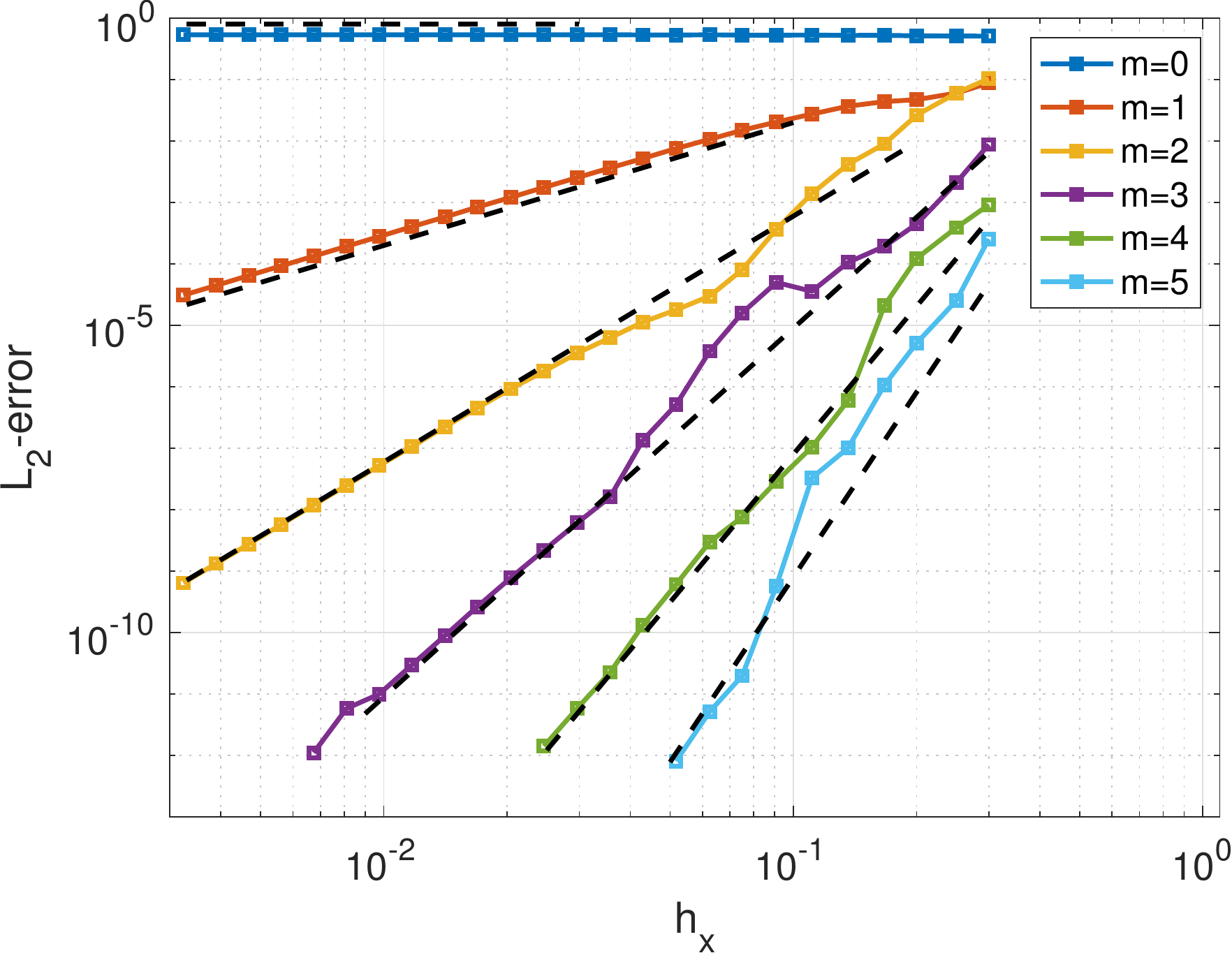}
  \includegraphics[width=0.45\textwidth]{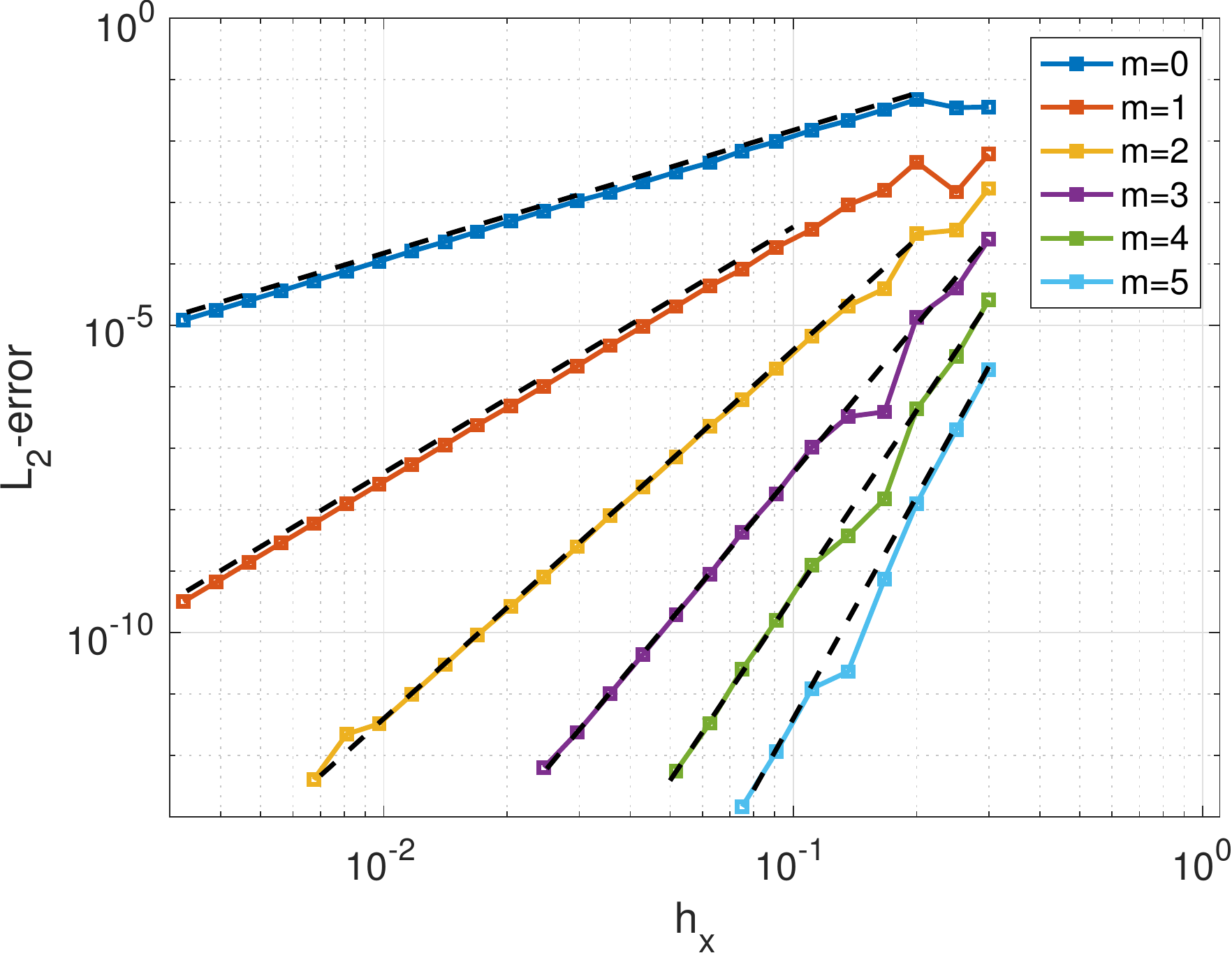}\\
  \caption{The conservative method, results are for the example with the Gaussian initial data. Displayed are $L_2$-errors as a function of $h_x$ for $m = 0,\ldots,5$. To the left is for $\lambda = 0.8$ and to the right   $\lambda = 1.0$. The dashed lines are $\sim  h_x^{2m}$ to the left and $\sim  h_x^{2m+2}$ to the right. \label{fig:cons_ex_1_errors}}
  \end{center}
\end{figure}

\subsection{Numerical Verification of Conservation}
The theory for the conservative method above does not account for roundoff effects and in this experiment we investigate how well the energy 
\[
 \mathcal{E}(t_n) \equiv \left\arrowvert P_{+}^{n} \right\arrowvert_{m+1}^2 + \left\arrowvert P_{-}^{n} \right\arrowvert_{m+1}^2,
 \]
is conserved between time steps. In the Figure \ref{fig:Pp_Pm} we display the $(m+1)$th derivative of $P_{+}$ and $P_{-}$ for the approximation of the exact solution $u(x,t) = \sin(x) \cos(t)$ at time $t = 0$ and on a grid $x\in [-\pi,\pi]$ with $h_x = \frac{2\pi}{3}$ and with $\Delta t = h_x / 2$. 

\begin{figure}[h!]
  \begin{center}
  \smallskip
  \includegraphics[width=0.50\textwidth]{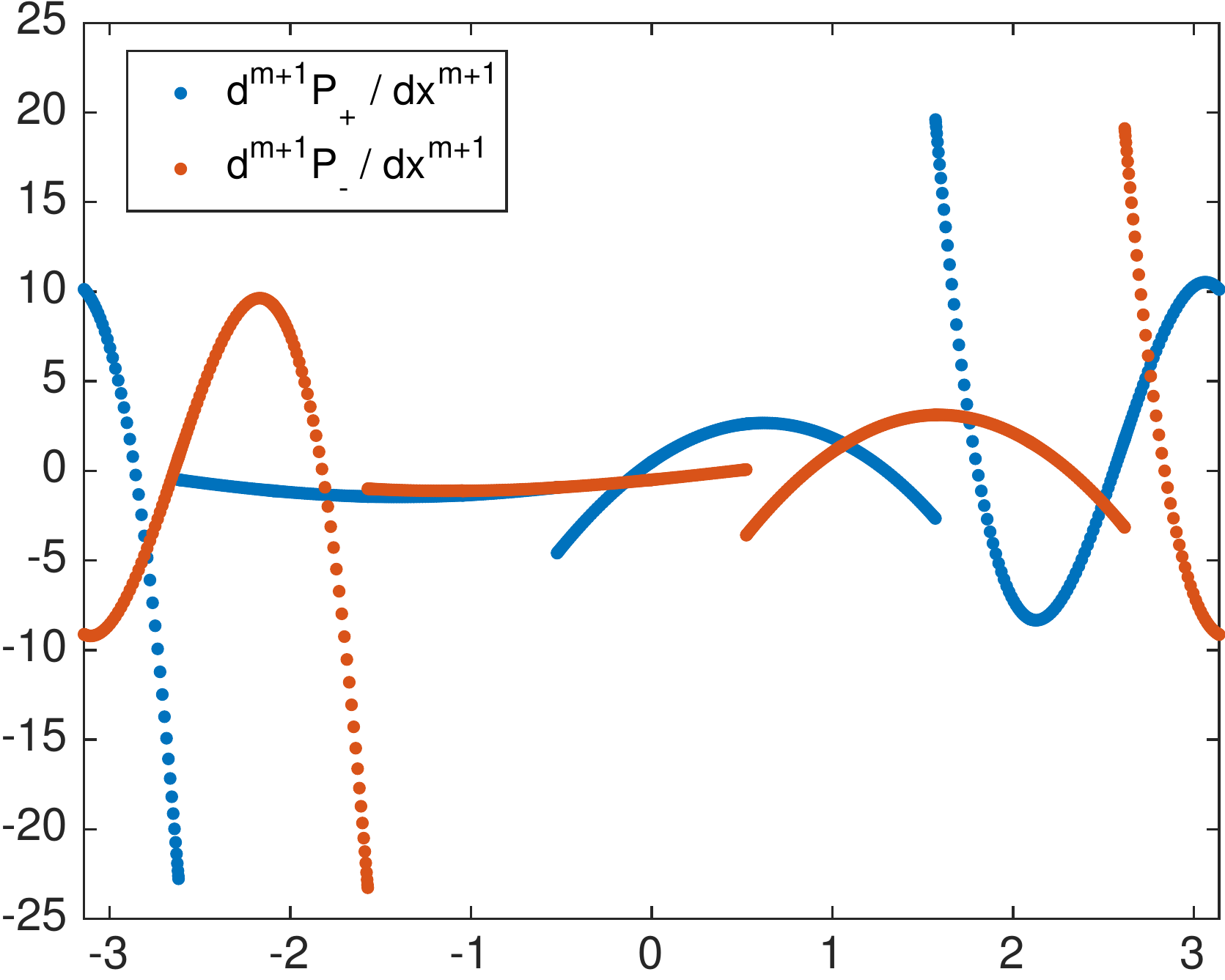}
  \caption{The figure displays the $(m+1)$th derivative of $P_{+}$ and $P_{-}$ for the approximation of the exact solution $u(x,t) = \sin(x) \cos(t)$ at time $t = 0$. \label{fig:Pp_Pm}}
  \end{center}
\end{figure}

As can be seen in the figure the $(m+1)$th derivatives of $P_{\pm}$ are discontinuous and displaced relative to the element centers. As the curves are piecewise smooth we (for simplicity) fix $\Delta t = h_x / 2$ which allows the sub-division of each element into four parts of size $h_x/4$ where the curves are continuous polynomials that can be integrated exactly (up to machine precision effects). 

In Figure \ref{fig:conservationsmooth_and _rand} we display the difference
in $\mathcal{E}$ at the initial time and subsequent time steps (we display every 100th time step). Two initial data are considered, the exact solution $u(x,t) = \sin(x) \cos(t)$, and random initial data. For the latter the initial polynomial coefficients are chosen randomly from a uniform distribution between 0 and 1. We use $h_x = 2\pi/ 30$ and monitor $\mathcal{E}$ for one million time steps for $m = 1,3,5$. In the first case the change in $\mathcal{E}$ is relatively small although the size of the difference is increasing with increasing $m$. For the second set of initial data the difference of $\mathcal{E}$ from its initial value increases more rapidly in $t$.

\begin{figure}[htb]
  \begin{center}
  \smallskip
  \includegraphics[width=0.32\textwidth]{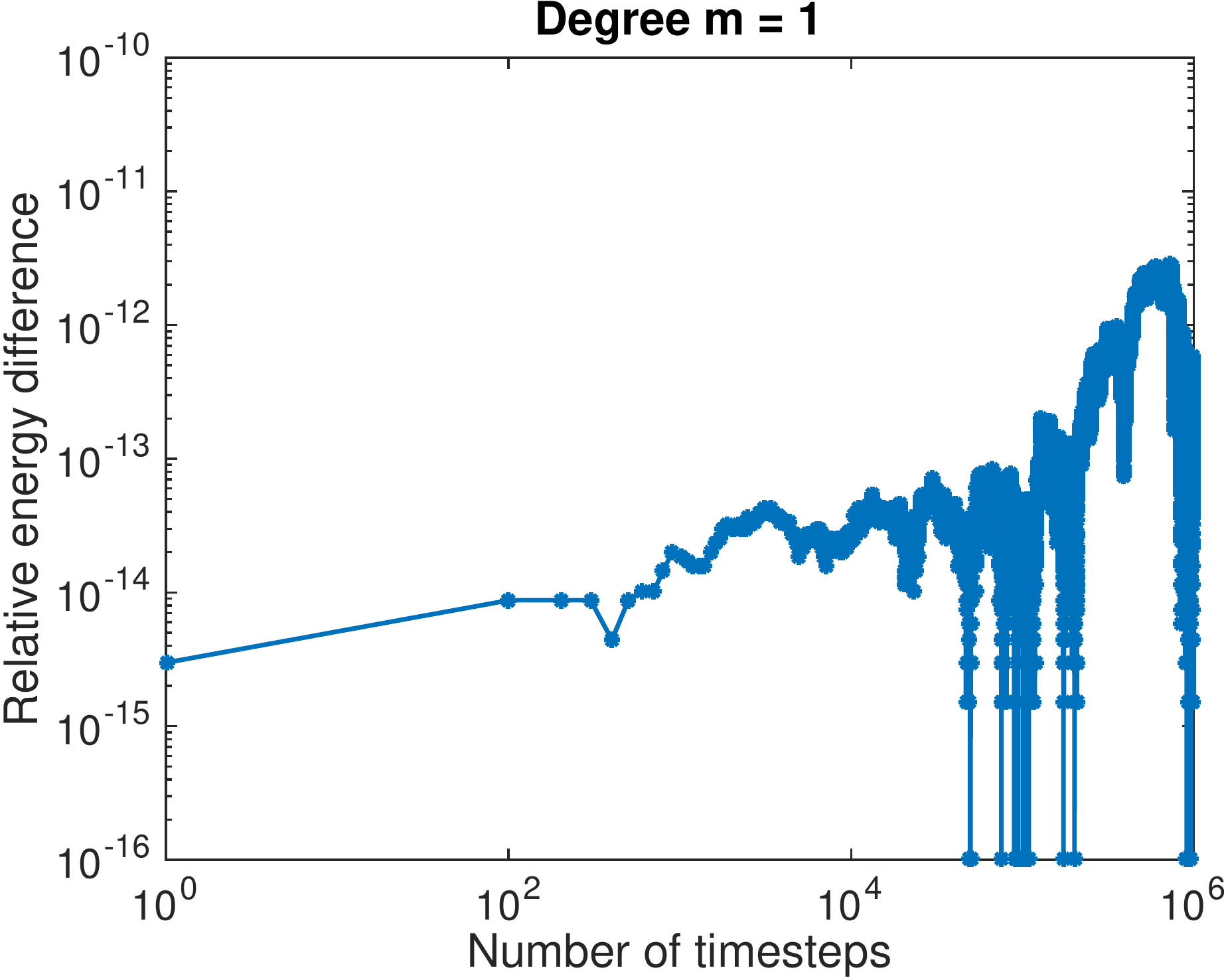}
  \includegraphics[width=0.32\textwidth]{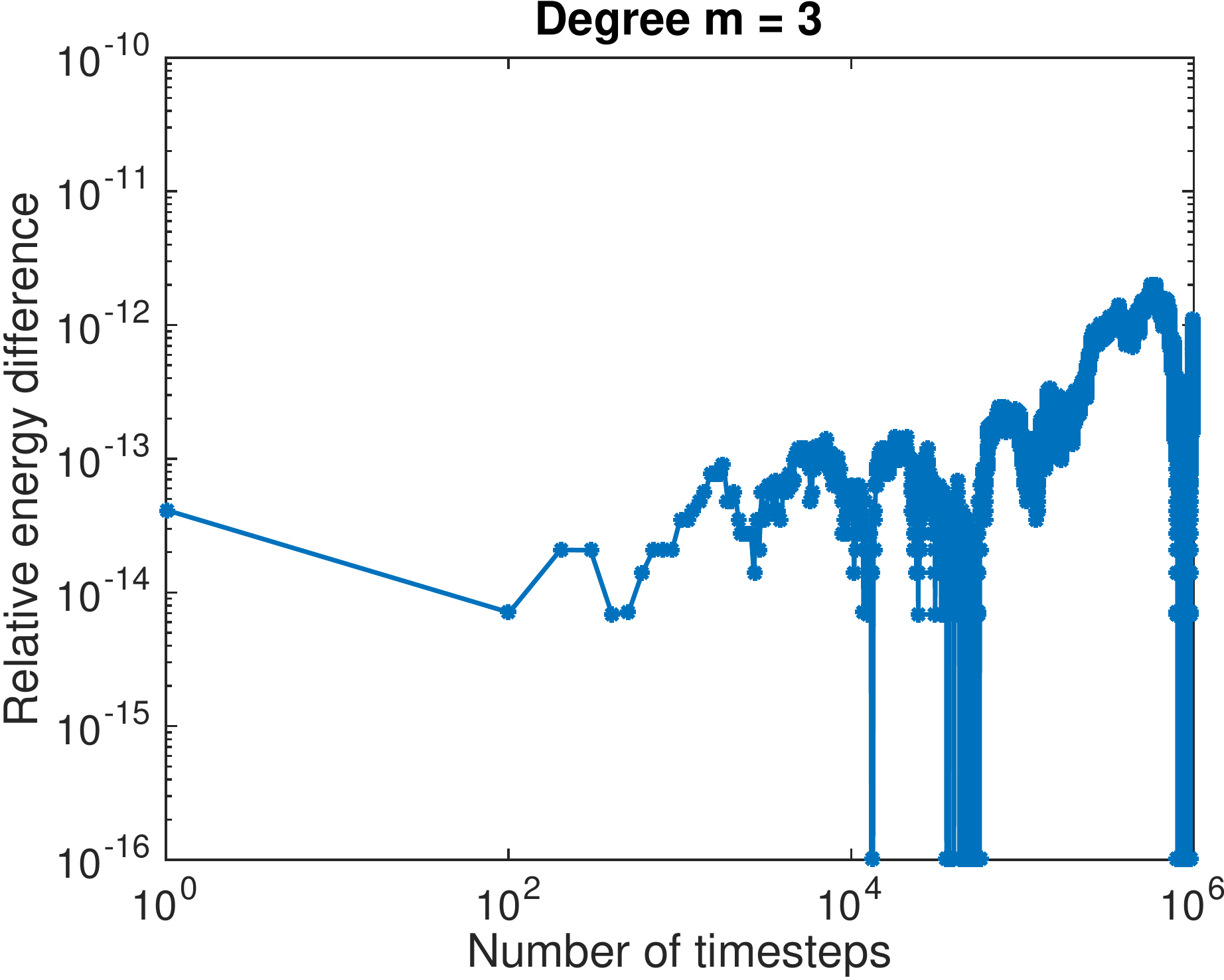}
  \includegraphics[width=0.32\textwidth]{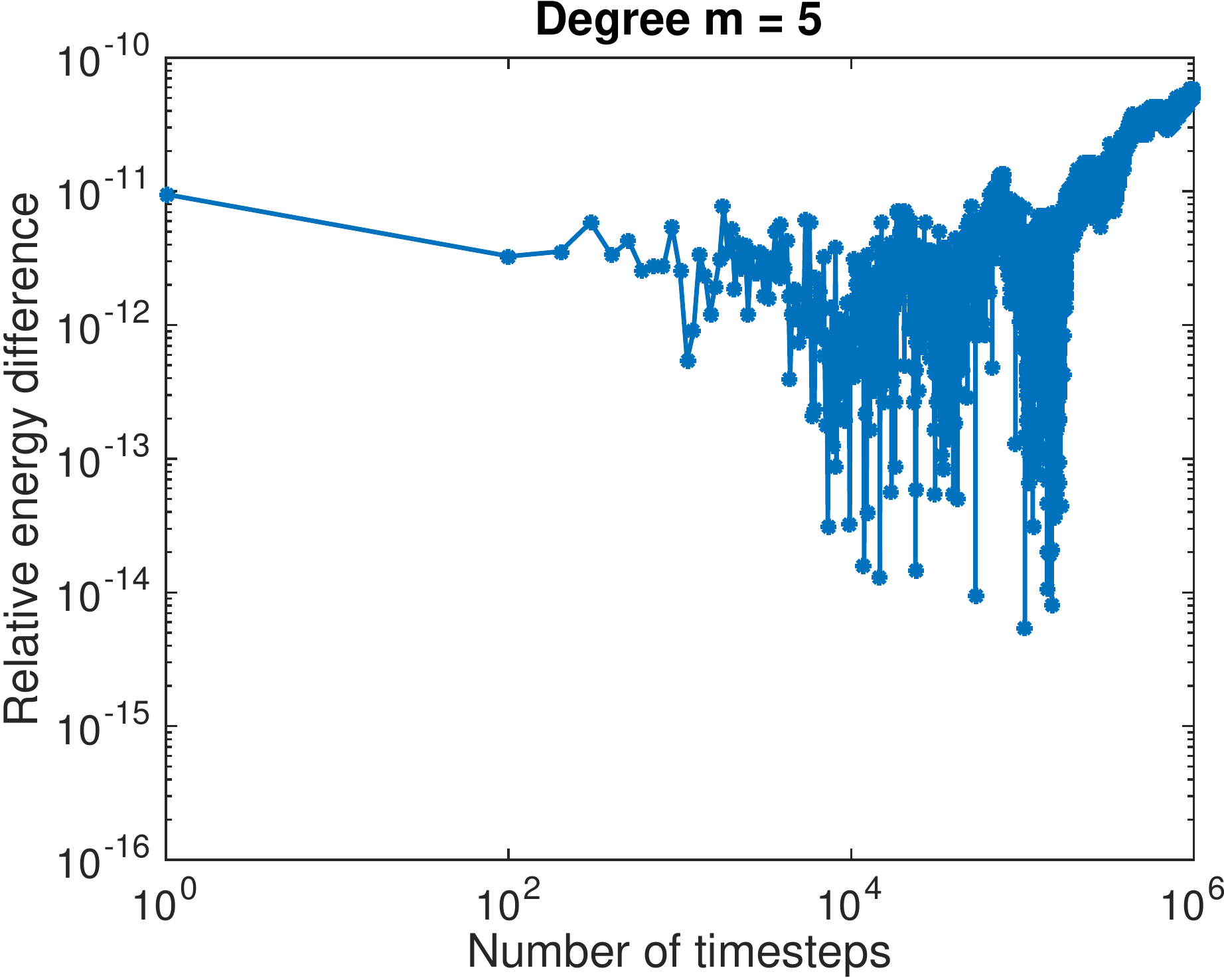}
  \smallskip
  \includegraphics[width=0.32\textwidth]{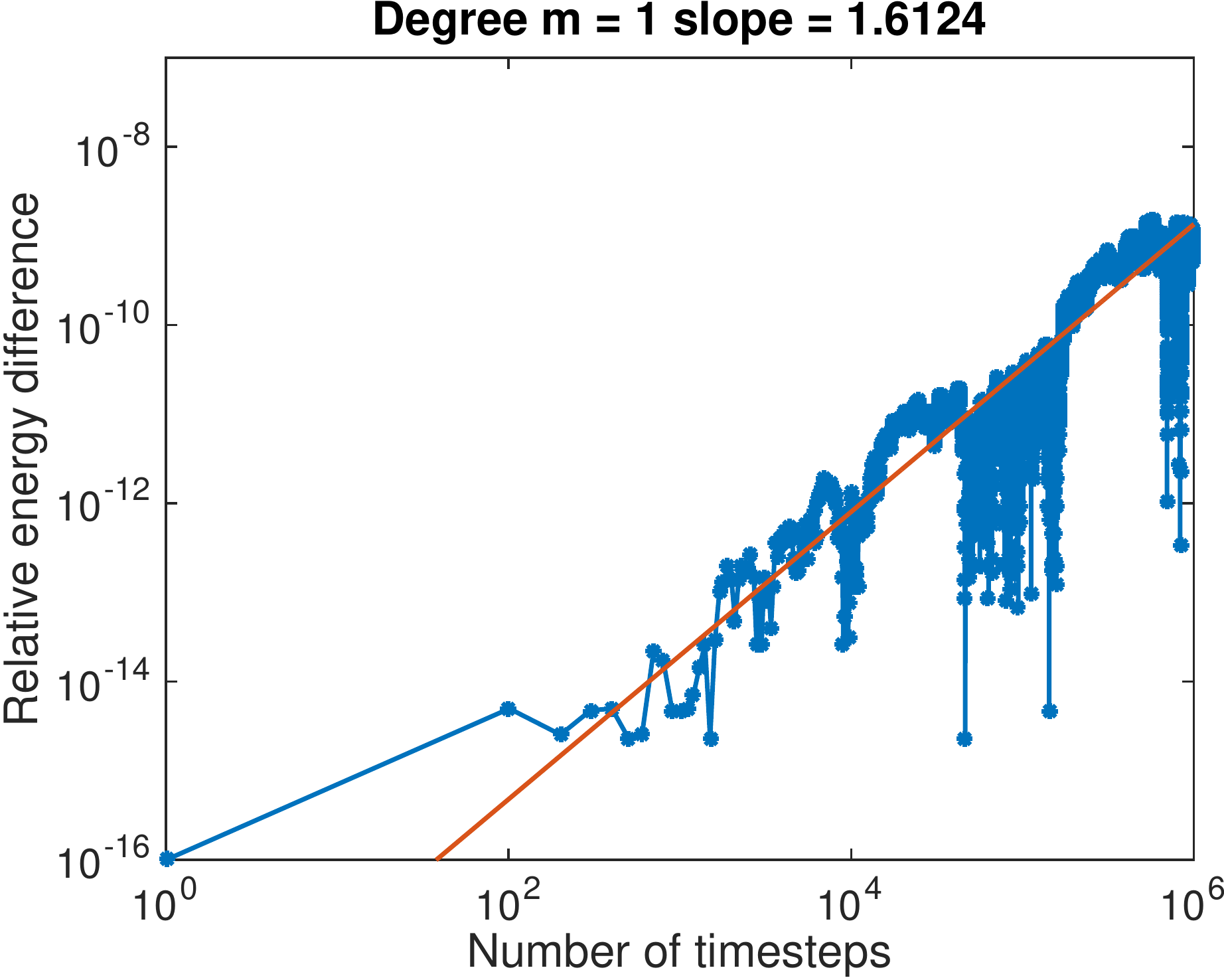}
  \includegraphics[width=0.32\textwidth]{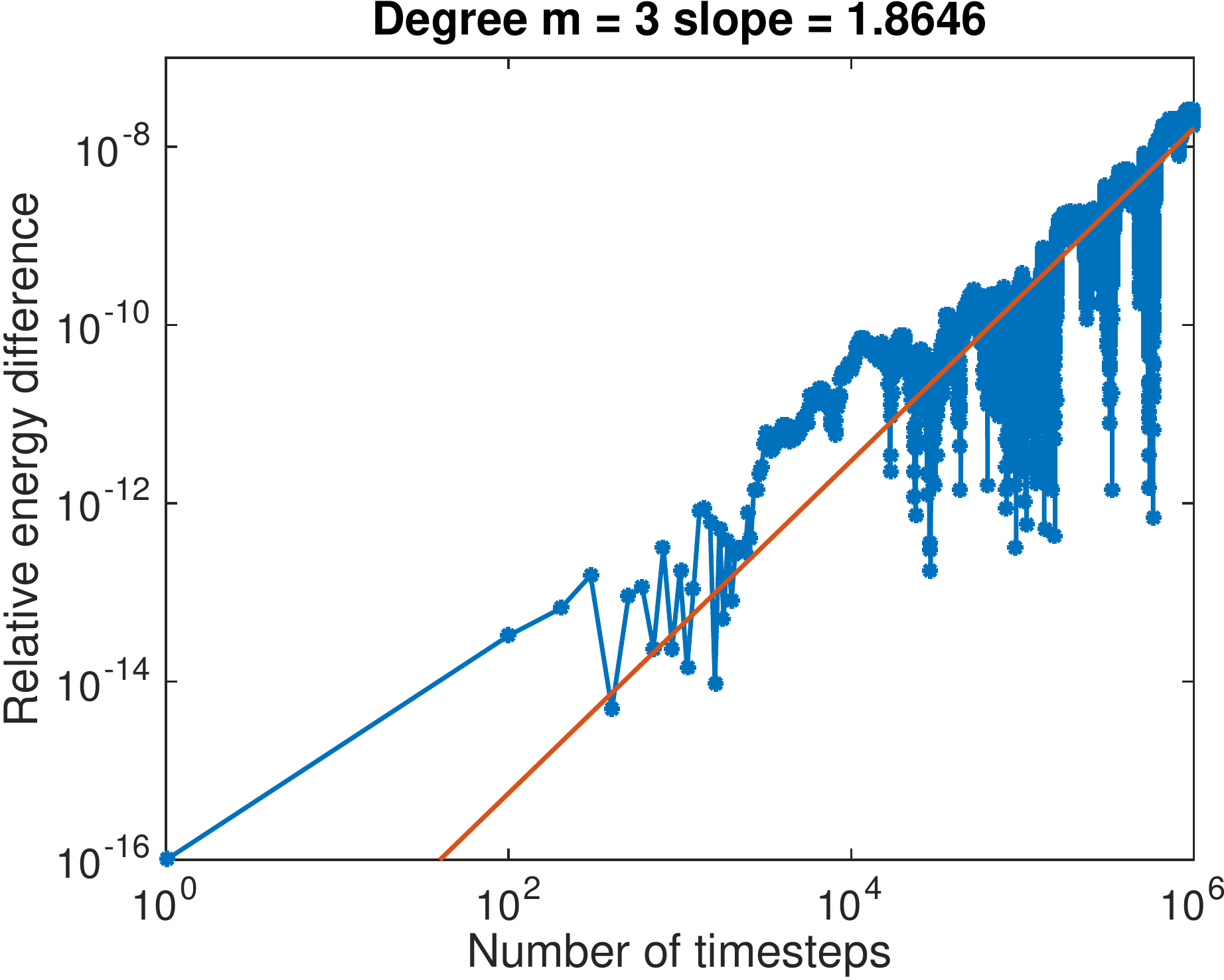}
  \includegraphics[width=0.32\textwidth]{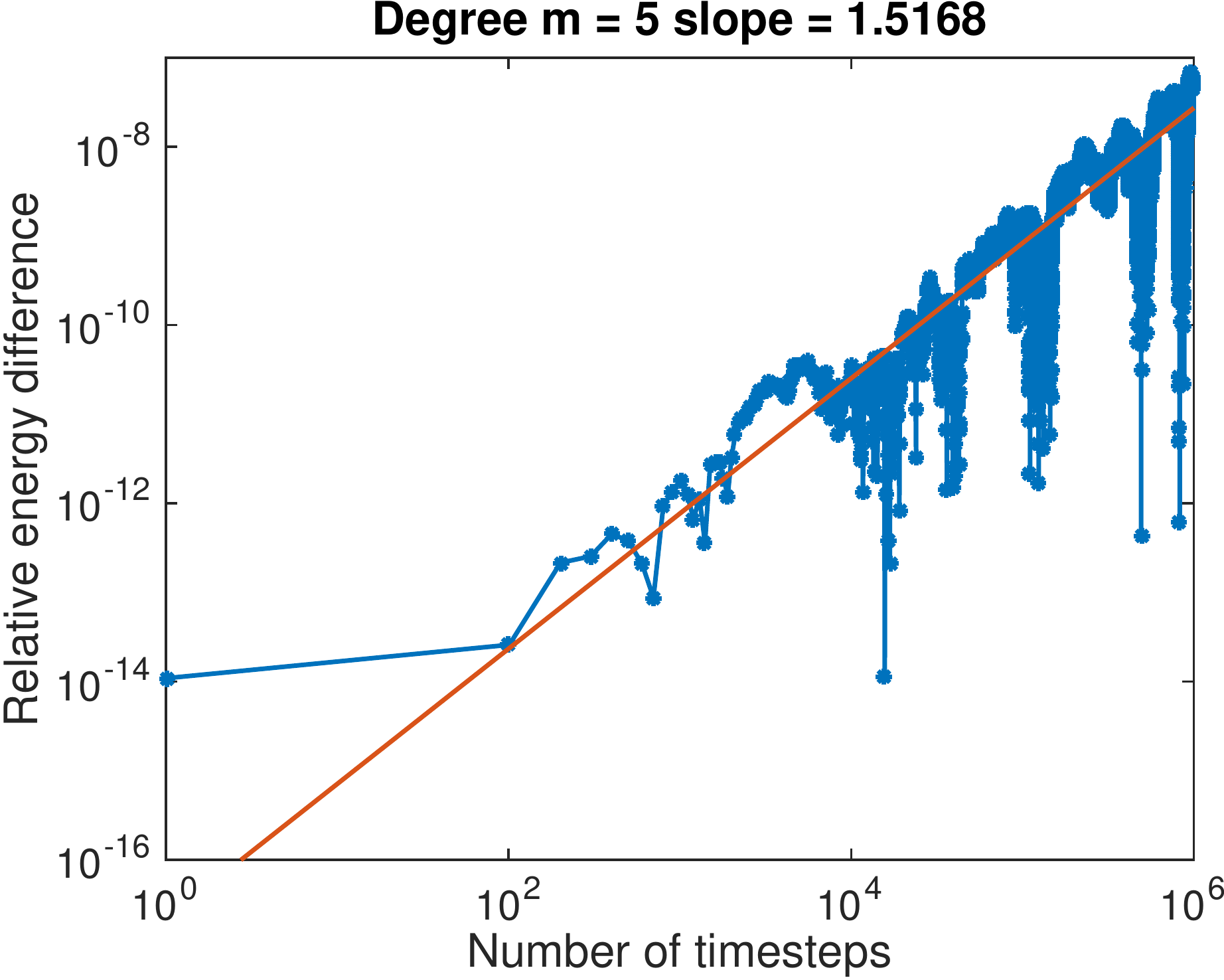}
  \caption{Very long time conservation properties for smooth and rough data under the influence of roundoff errors. \label{fig:conservationsmooth_and _rand}}
  \end{center}
\end{figure}

\subsection{An Example in Two Space Dimensions}
To verify the convergence properties of the two methods in two space dimensions we approximate the plane wave $u = \sin ( 2 \pi \kappa (x + y + \sqrt{2} t)$ on the unit square with periodic boundary conditions. As in the one dimensional case we fix $\lambda = 0.8$ and $1.0$ and take the final time to be as close to $4.18$ as possible using an integer number of time steps. The parameter $\kappa$ is taken to be $m+1$.

In Figures \ref{fig:2d_errors_diss} and \ref{fig:2d_errors_cons} we display the $L_2$-errors in $u$ for the conservative and dissipative methods. As in the one dimensional case we see that the methods are of order $2m$ and $2m-1$ when $\lambda <1$ and $2m+2$ $2m$ respectively when $\lambda=1.0$.

\begin{figure}[]
  \begin{center}
  \includegraphics[width=0.45\textwidth]{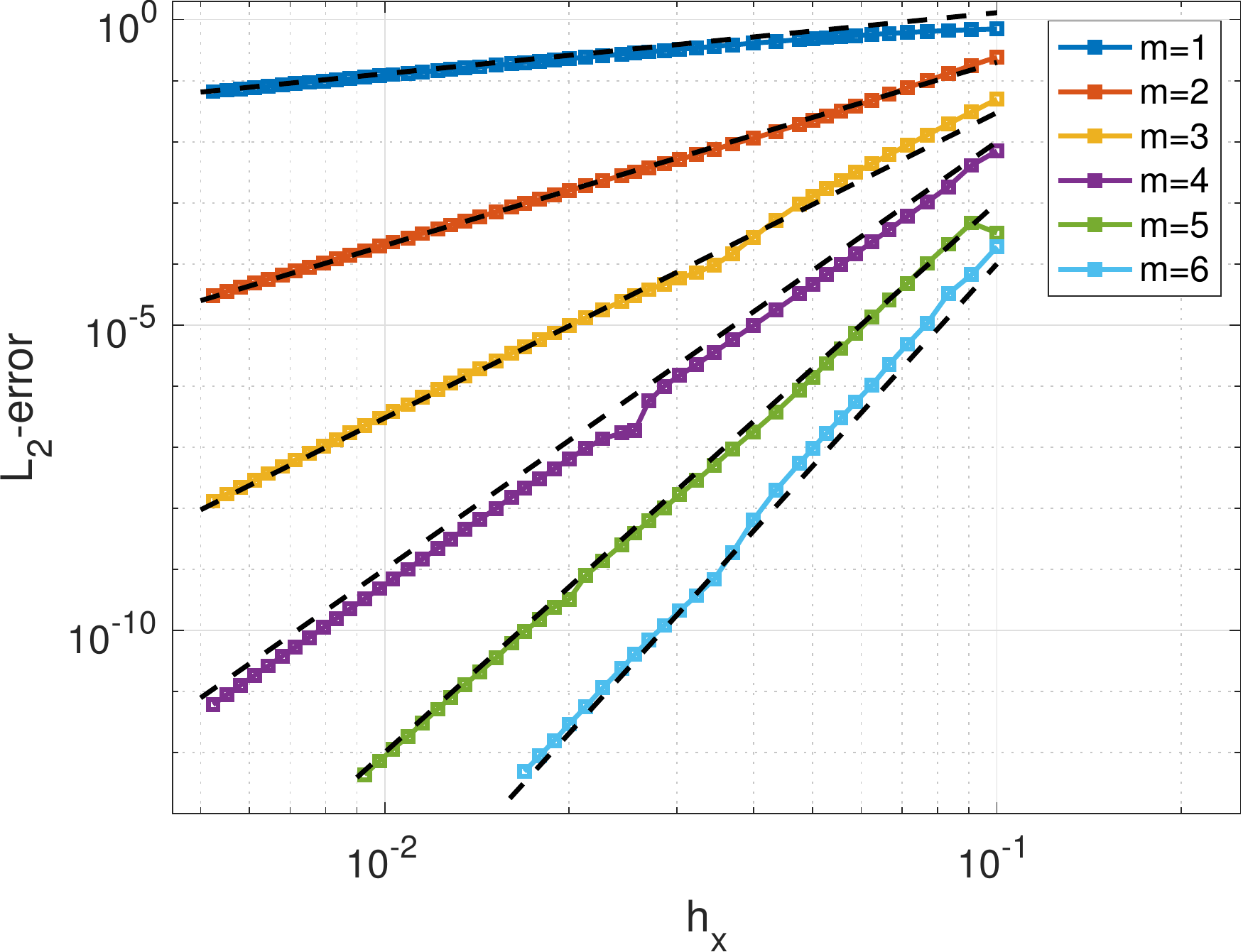}\ \ 
  \includegraphics[width=0.45\textwidth]{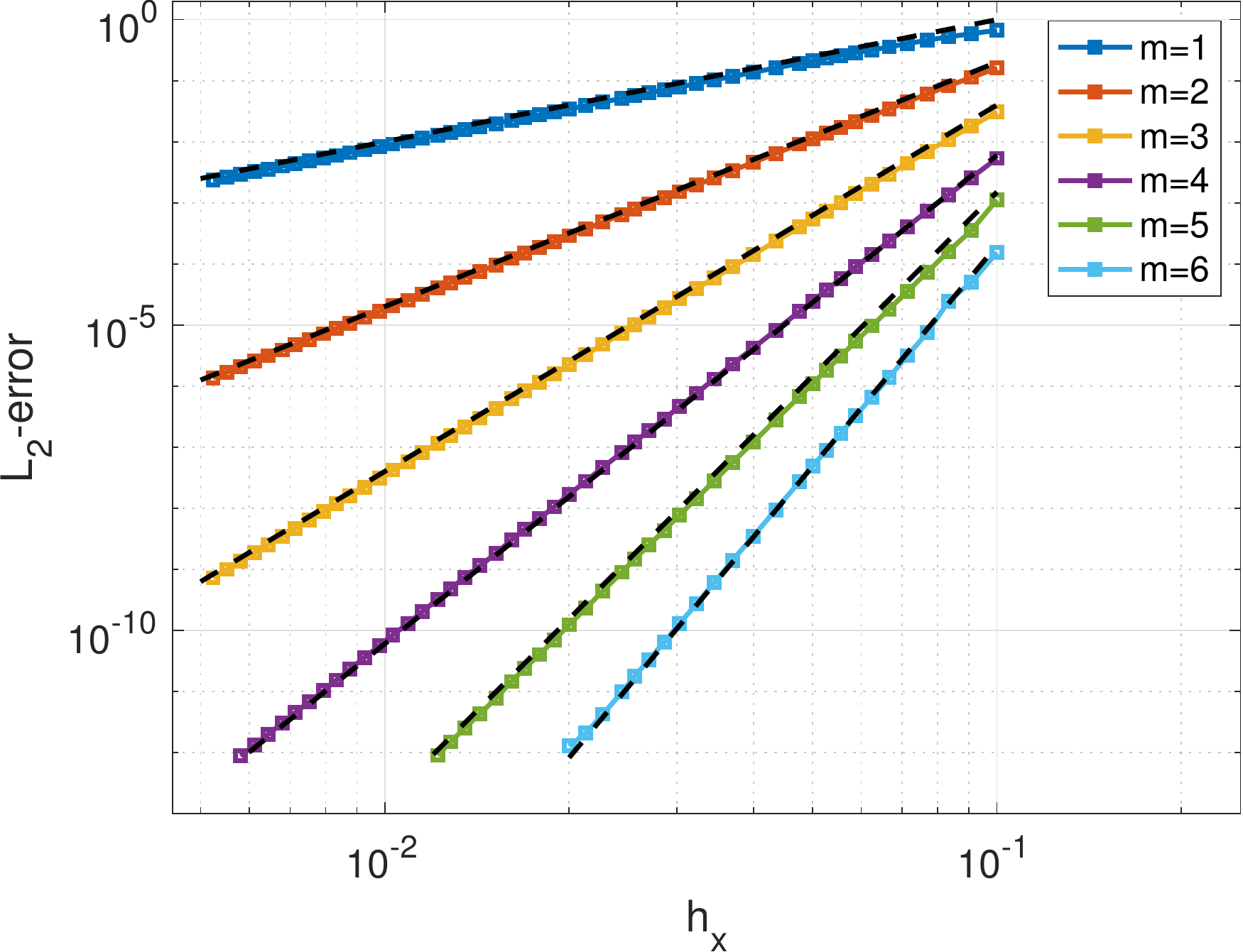}\\
  \caption{Displayed are the $L_2$-errors as a function of grid size for $m = 1,\ldots,6$ for $\lambda = 0.8$ to the left and $\lambda = 1.0$ to the right. These results are for the dissipative method and the dashed slopes are $\sim h_x^{2m-1}$ and $\sim h_x^{2m}$ for $\lambda = 0.8$ and $\lambda = 1.0$. \label{fig:2d_errors_diss}}
  \end{center}
\end{figure}

\begin{figure}[]
  \begin{center}
  \includegraphics[width=0.45\textwidth]{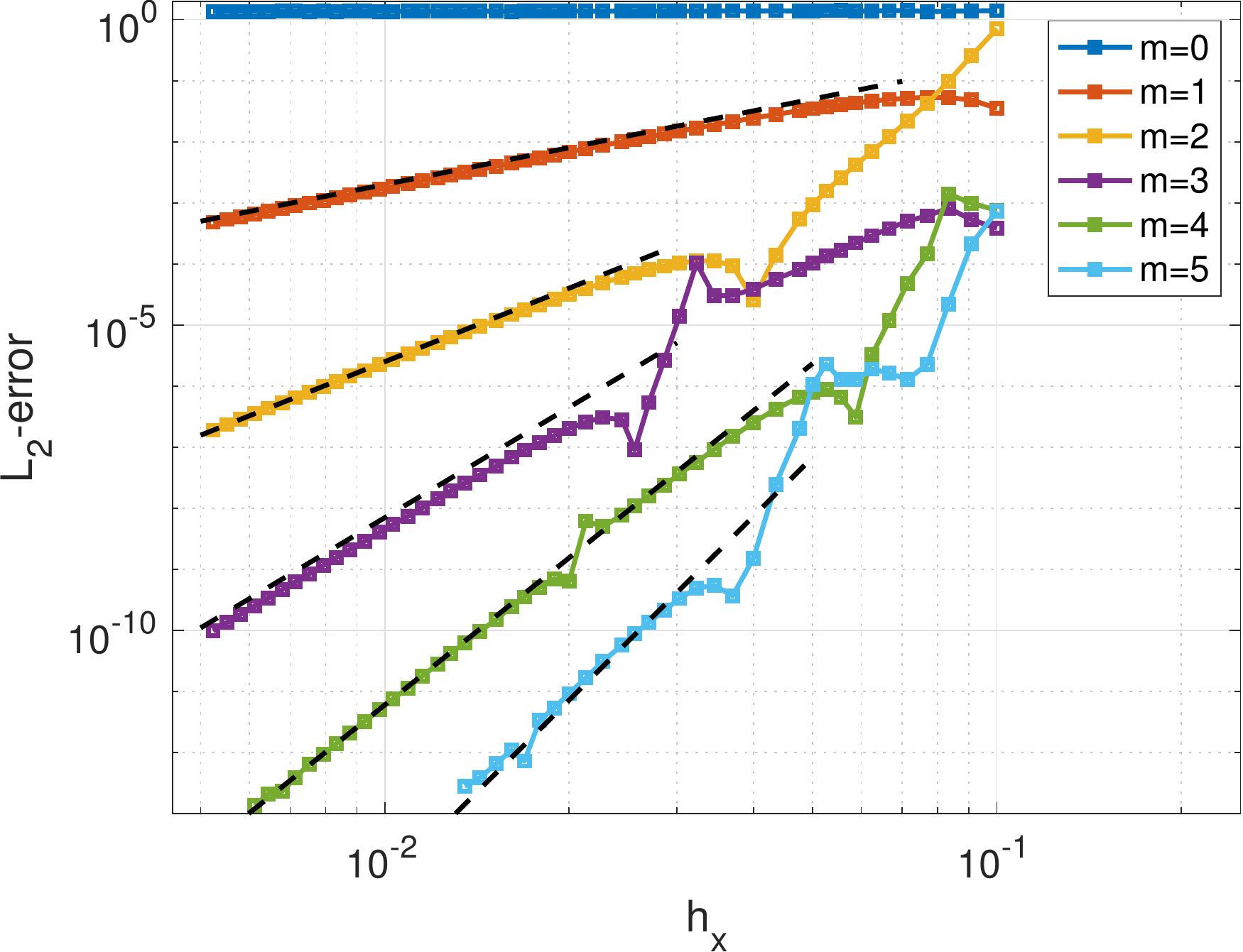}\ \ 
  \includegraphics[width=0.45\textwidth]{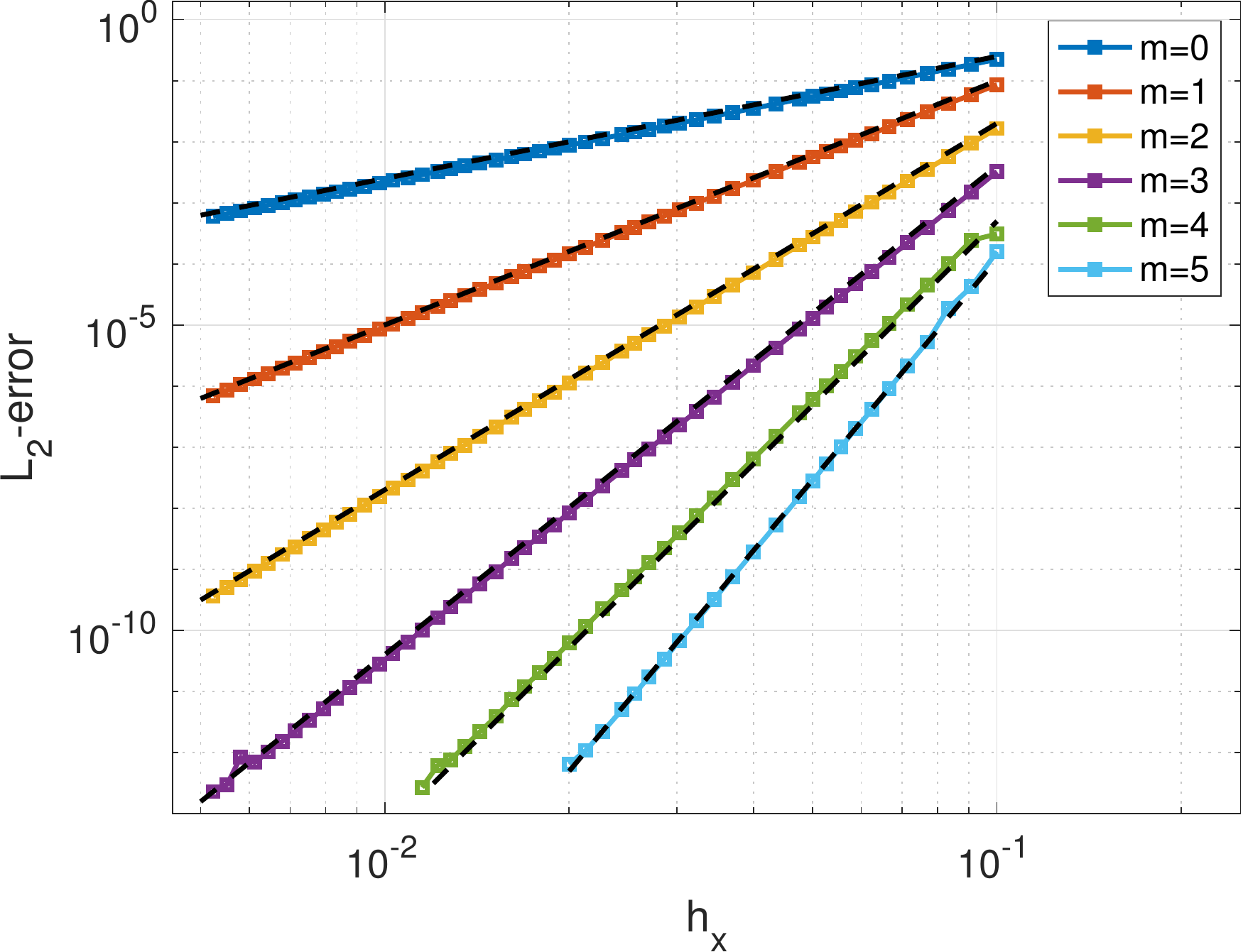}\\
  \caption{Displayed are the $L_2$-errors as a function of grid size for $m = 1,\ldots,6$ for $\lambda = 0.8$ to the left and $\lambda = 1.0$ to the right. These results are for the conservative method and the dashed slopes are $\sim h_x^{2m}$ and $\sim h_x^{2m+2}$ for $\lambda = 0.8$ and $\lambda = 1.0$. \label{fig:2d_errors_cons}}
  \end{center}
\end{figure}

\section{Implementations on Accelerators}\label{sec:occa}
The Hermite methods of Goodrich and co-authors were first tailored to graphics processing units by Dye in \cite{dye_thesis} wherein an implementation for the two dimensional advection equation was proposed. Strategies for three-dimensional equations were later proposed by Vargas and co-authors in \cite{vargas2016gpu}. Here we focus on extending the techniques introduced in \cite{vargas2016gpu} to develop algorithms for the dissipative and conservative Hermite methods in three space dimensions.

In \cite{vargas2016gpu} two strategies for implementing Hermite methods on GPUs were explored. The first approach proposed separate interpolation and evolution kernels enabling specialized tuning of each step. A drawback of this approach is the additional memory required to explicitly store the interpolant. To bypass the need of additional storage, a monolithic kernel was proposed which fused the interpolation and evolution procedures into a single kernel call. Performance results in \cite{vargas2016gpu} demonstrated that both approaches led to a comparable time to solution. 

Keeping consistent with the CUDA nomenclature \cite{Sanders:2010:CEI:1891996}, we expose two levels of parallelism. The first layer assigns a block of threads to carryout the interpolation and evolution procedure on a cell. The second layer exposes fine grained parallelism in which block local threads carry out the local operations found in constructing the interpolant and evolution step. Exploiting the tensor product structure of the polynomials, the interpolants are reconstructed in a dimension by dimension manner and expressed as a series of matrix-matrix multiplications. To maximize performance, shared memory is used to hold intermediate computations (shared memory is replaced with memory on the stack when considering the CPU). Additionally, we apply the shared memory rolling technique as used in \cite{vargas2016gpu, vargas2017hermite}.

The experiments presented in this manuscript were carried out on a single node of SMU's ManeFrame II compute cluster using double precision and an NVIDIA Tesla P100 graphics card as the accelerator. The card has a theoretical peak bandwidth of 732 GB/s and can potentially perform 4700 Gigaflops in double precision. To estimate a more ``realistic'' bandwidth we consider a simple DAXPY (combination of scalar multiplication and vector addition) kernel. For a sufficiently large vector a streaming bandwidth of 552 GB/s was observed on the graphics card; although this is not indicative of the card's capabilities it does serve as a representative of `` achievable peak'' performance numbers.

In this work we evaluate kernel performance by measuring the effective arithmetic throughput and effective memory bandwidth using the NVIDIA profiler. The effective arithmetic throughput is computed using the number of floating-point operations given by flop\_count\_dp, and the effective memory bandwidth is obtained by summing \textit{dram\_read\_throughput} and \textit{dram\_write\_throughput}. Here bandwidth corresponds to the sum of bytes read and written to global memory by a GPU kernel. We adopt OCCA as our API \cite{medina2015okl} which allows us to cross-compile code into a variety of API's. For completeness we provide a comparison of time to solution for a GPU and a  multi-core CPU. The multi-core CPU is a dual 18-core Intel Xeon CPU E5-2695 with an operating frequency of 2.10 GHz.

We note that these experiments are for problems posed in three space dimensions, as it seems most important to develop efficient implementations of the method with an eye towards challenging large-scale problems. 

\subsection{Dissipative Method}
Following the techniques described above, we begin with tailoring the dissipative Hermite method to the graphics processing unit. The main challenge in tailoring the method stem from the varying number of degrees of freedom used in representing the displacement and velocity variables.  This becomes troubling as building interpolants of varying spatial order leads to irregular access patterns. To overcome this challenge local arrays are padded with zeros to ensure data structures of equal dimensions thus simplifying the implementation.  The first version of the method, which will be referred to as the \textbf{simple-method}, introduced in this article reconstructs  a single Hermite interpolant for the displacement and velocity components respectively; while the second variant, which will be referred to as \textbf{full-method} introduces an additional interpolant for each spatial dimension (total of five interpolants), the additional interpolants were necessary in order to maintain $\Delta t \approx  \min (h_x, h_y, h_z)$. In this section we explore both version and we compare time per iteration. 

We begin our performance analysis by considering the two kernel approach for the simple-method and follow by considering monolithic kernels for both the simple and full methods. As the full method would require substantial global memory for the five interpolants we did not consider the two kernel approach. Figure \ref{fig:DissipativeA} reports the observed bandwidth and GFLOPs for both versions of the method. For these numerical experiments, we chose the number of stages of the Hermite-Taylor Scheme to be $2m+2$, which is less than what is required for exact evolution of the cell polynomials, but sufficient for convergence at design order. As intermediate computations are stored in shared memory we find that the second variant ( with additional interpolants) required more shared memory than supported in the graphics card and thus we cannot go beyond an $m=2$ method. Here we may conclude that by considering two kernels we can achieve a significantly higher dram throughput but not a significantly better time to solution.

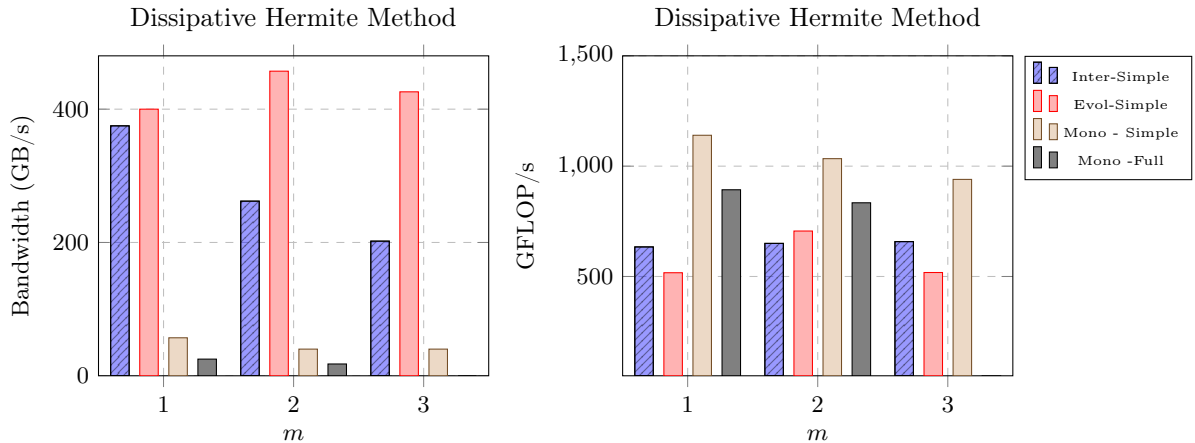
\begin{figure}[h!]
\centering
\subfloat{
\begin{tikzpicture}
\begin{axis}[
	legend style={font=\tiny},
	width=.41\textwidth,
	xmin=0.5,xmax=3.5,
	xtick={1,2, 3},
	ymin=0,
	ymax=480,
	ylabel=Bandwidth (GB/s),
	xlabel=$m$,
	ybar=4pt,
	bar width=7pt,
	xmajorgrids=true,
	ymajorgrids=true,
	grid style=dashed,
	legend pos=north west,
	title=Dissipative Hermite Method
]

\addplot [fill=blue, fill opacity=0.4, postaction={pattern=north east lines}]
	coordinates {(1,375) (2,262) (3,202)}; 	

\addplot
	coordinates {(1,400) (2,457) (3,426)}; 

\addplot 
	coordinates {(1,57) (2,40) (3,40)}; 

\addplot
	coordinates {(1,25) (2,17.7) (3,0)};



\end{axis}
\end{tikzpicture}
}
\subfloat{
\begin{tikzpicture}
\begin{axis}[
	legend style={font=\tiny},
    xmin=0.5,xmax=3.5,
	width=.41\textwidth,
	xtick={1,2, 3},
	ylabel=GFLOP/s,
	xlabel=$m$,
	ymin=50,
	ymax=1500,
	ybar=4pt,
	bar width=7pt,
	xmajorgrids=true,
	ymajorgrids=true,
	legend pos=north west,
	grid style=dashed,
	title=Dissipative Hermite Method,
	legend pos=outer north east
]

\addplot [fill=blue, fill opacity=0.4, postaction={pattern=north east lines}]
	coordinates {(1,634) (2,650) (3,658)}; 	

\addplot
	coordinates {(1,517) (2,706) (3, 518)}; 

\addplot
	coordinates {(1,1140) (2,1034) (3,940)};

\addplot
	coordinates {(1,893) (2,834) (3,0)}; P-100

\legend{Inter-Simple, Evol-Simple, Mono - Simple,Mono -Full}
\end{axis}
\end{tikzpicture}
}
\caption{Performance of the Dissipative Hermite Method.\label{fig:DissipativeA}}
\end{figure}

\begin{table}[h!]
\centering
\begin{tabular}{|l|c|c|c|c|}
\hline
Choice of $m$        &  $m=1$  & $m=2$ & $m=3$ \\ \hline
Simple - Interpolation &  375 / 634 & 262 / 650   & 202 / 658    \\ \hline
Simple - Evolution   & 400 / 517 & 457 / 706  & 426 / 518    \\ \hline
Simple - Monolithic   & 57 / 1140 & 40 / 1034  & 40 / 940   \\ \hline
Full      - Monolithic   & 25 / 893  & 17.7 / 834  &  -   \\ \hline\end{tabular}
\caption{Table of Bandwidth (GB/s) and GFLOP/s. }
\label{table:Bandwidth/GFLOPS}
\end{table}


\subsubsection{Time per Iteration}
For completeness, a comparison of time per iteration is presented. Here we tailor separate OCCA kernels for the GPU and CPU. Table \ref{table:timeToSolutionA}  reports the time per iteration when targeting different platforms via the OCCA API. Obviously the
GPU delivers the solution in substantially shorter time. 


\begin{table}[h!]
\centering
\begin{tabular}{|l|c|c|c|c|}
\hline
Choice of $m$       &  $m=1$  & $m=2$ & $m=3$\\ \hline
OCCA::OpenMP - Simple Method  &   0.88 sec & 1.16 sec  &  1.92 sec \\ \hline
OCCA::OpenMP - Full Method  &  1.50 sec  & 2.10 sec & 2.57  sec \\ \hline
OCCA::CUDA  - Single -Simple method &  0.057 sec & 0.087   sec & 0.10   sec \\ \hline
OCCA::CUDA - Two Kernels - Simple method   & 0.15  sec &  0.15 sec &  0.156 sec \\ \hline
OCCA::CUDA - Full method   & 0.11 sec &  0.17  sec &  - \\ \hline
\end{tabular}
\caption{Comparison of time per iteration of the dissipative Hermite kernels executed on the GPU and CPU. For orders $m=1,2, 3$ the number of grid points were chosen to be 160,  120, 90  per Cartesian direction.}\label{table:timeToSolutionA}
\end{table}

\subsection{Conservative Hermite Method}
Building on the point-wise formula given in Equation (\ref{eq:convUpdate}) we consider an algorithm which 
first approximates high order derivatives via Hermite interpolation, and propagates the solution by accumulating the necessary spatial derivatives.  Figure \ref{fig:nonDissipative} reports the peak achieved bandwidth and GFLOPs for both the separate interpolation/evolution and monolithic kernels. Similar to previous experiments, we find that the performance of separate interpolation and evolution kernels is limited by the device's ability to read and write to global memory.  The monolithic kernel does not experience this bottleneck since it does not explicitly write out the interpolant.

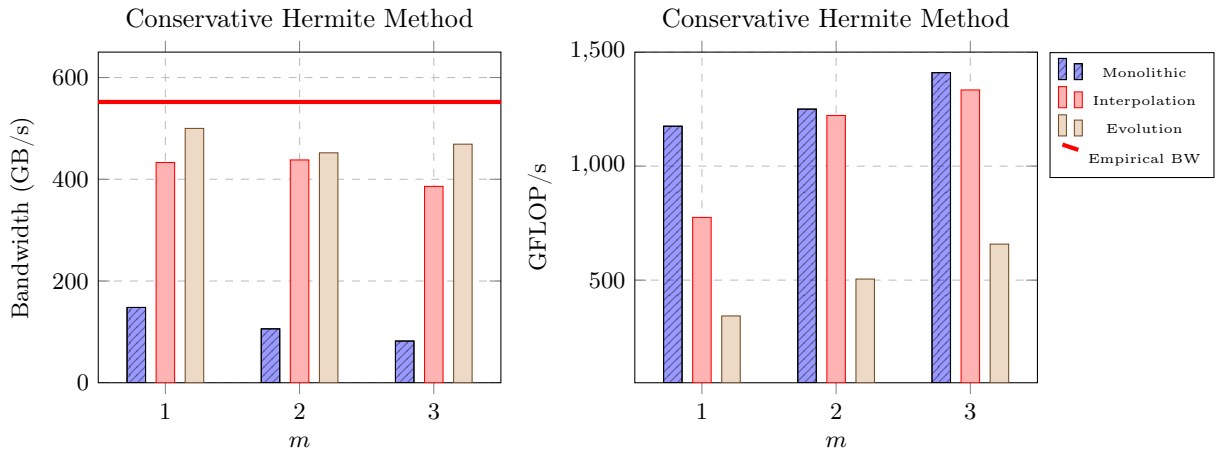
\begin{figure}[h!]
\centering
\subfloat{
\begin{tikzpicture}
\begin{axis}[
	legend style={font=\tiny},
	width=.42\textwidth,
	xmin=0.5,xmax=3.5,
	xtick={1,2,3},
	ymin=0,
	ymax=650,
	ylabel=Bandwidth (GB/s),
	xlabel=$m$,
	ybar=4pt,
	bar width=7pt,
	xmajorgrids=true,
	ymajorgrids=true,
	grid style=dashed,
	legend pos=north west,
	title=Conservative Hermite Method
]

\addplot[fill=blue, fill opacity=0.4, postaction={pattern=north east lines}]
	coordinates {(1,148) (2,106) (3,82)}; 

\addplot
	coordinates {(1,433) (2,438) (3,386)}; 
	
\addplot
	coordinates {(1,500) (2,452) (3,469)};
	
\addplot[red,sharp plot,update limits=false, style=ultra thick] coordinates {(0,552) (10,552)};


\end{axis}
\end{tikzpicture}
}
\subfloat{
\begin{tikzpicture}
\begin{axis}[
	legend style={font=\tiny},
    xmin=0.5,xmax=3.5,
	width=.42\textwidth,
	xtick={1,2,3},
	ylabel=GFLOP/s,
	xlabel=$m$,
	ymin=50,
	ymax=1500,
	ybar=4pt,
	bar width=7pt,
	xmajorgrids=true,
	ymajorgrids=true,
	legend pos=north west,
	grid style=dashed,
	title=Conservative Hermite Method,
	legend pos=outer north east
]
\addplot[fill=blue, fill opacity=0.4, postaction={pattern=north east lines}]
	coordinates {(1,1175) (2,1250) (3,1410)};

\addplot
	coordinates {(1,775) (2,1222) (3,1334)};

\addplot	
	coordinates {(1,343) (2,504.78) (3,658)}; 

\addplot[red,sharp plot,update limits=false, style=ultra thick] coordinates {(0,0) (0,0)};
	
\legend{Monolithic,Interpolation,Evolution, Empirical BW}
\end{axis}
\end{tikzpicture}
}
\caption{Performance for the conservative Hermite method.\label{fig:nonDissipative}}
\end{figure}

\begin{table}[h!]
\centering
\begin{tabular}{|l|c|c|c|c|}
\hline
Choice of $m$        &  $m=1$  & $m=2$ & $m=3$ \\ \hline
Monolithic Kernel  &  148/1175  & 106/1250  &  82/1410 \\ \hline
Interpolation Kernel   & 433/775 &  438/1222 &  386/1334  \\ \hline
Evolution Kernel    & 500/343 & 452/504  & 469/658 \\ \hline
\end{tabular}
\caption{Table of Bandwidth (GB/s) and GFLOP/s. \label{table:Bandwidth/GFLOPSA}}
\end{table}

\subsubsection{Time per Iteration}
For completeness, a comparison of time per iteration is presented. As before we tailor separate OCCA kernels for the GPU and CPU. Table \ref{table:cpuvgpuRuntime}  reports the time per iteration when targeting different platforms via the OCCA API. Again the time to
solution is significantly smaller for the GPU. 

\begin{table}[h!]
\centering
\begin{tabular}{|l|c|c|c|c|}
\hline
Choice of $m$        &  $m=1$  & $m=2$ & $m=3$ \\ \hline
OCCA::OpenMP  &  0.55 sec &  0.86 sec &  1.12 sec \\ \hline
OCCA::CUDA - Single Kernel   & 0.035 sec &  0.052 sec &  0.0579 sec \\ \hline
OCCA::CUDA - Two Kernels    & 0.057sec &  0.0628 sec &  0.0619 sec\\ \hline
\end{tabular}
\caption{Comparison of time per iteration of Hermite kernels executed on the GPU and CPU. For orders $m=1,2,3$ the number of grid points were chosen to be 280, 190, 140 points per Cartesian direction. Noticeably a single GPU kernel offers a slightly better time to solution due to the advantage of not having to store the interpolant.\label{table:cpuvgpuRuntime}}

\end{table}


\section{Conclusion and Future Work}

We have demonstrated two approaches to solving the scalar wave equation using Hermite interpolation. Both are stable for CFL up to
$1$ independent of order and converge at a rate roughly twice the number of degrees-of-freedom per node. Moreover, we find that the
methods do an excellent job of exploiting GPUs, with time-to-solution on a single NVIDIA P100 twenty to thirty times less than on a
dual 18-core CPU.

From a practical perspective, future work must include implementations in complex geometry; we plan to use both overset
structured grids and hybrid structured-unstructured grids. In addition we will generalize the method to systems, such as the
elastic wave equation. Theoretically, although one can prove that the methods are stable at full CFL,
we have as yet been unable to extend the convergence analyses to multiple space dimensions.

\bibliographystyle{plain}
\bibliography{appelo,hag}
\end{document}